\documentclass[12pt]{amsart}

\usepackage[usenames,dvipsnames]{xcolor}
\usepackage{amsthm}
\usepackage{amsfonts}
\usepackage{fancyhdr}
\usepackage{tikz}
\usepackage{hyperref}
\usepackage[normalem]{ulem}
\usepackage{graphicx}
\usepackage{caption}
\usepackage{cite}
\usepackage{subcaption}
\usepackage{amsmath,amscd}
\newcommand{\spl}[2]{{\bf #1}_{#2}}

\newtheorem{theorem}{Theorem}[section]
\newtheorem{corollary}[theorem]{Corollary}

\newtheorem{proposition}[theorem]{Proposition}

\newtheorem{definition}[theorem]{Definition}
\newtheorem{remark}[theorem]{Remark}
\newtheorem{example}[theorem]{Example}

\newtheorem{question}[theorem]{Question}

\begin{document}
\date{}
\title{Generalized Splines on Arbitrary Graphs}
\thanks{We are very grateful for the contributions of Arielle McCoy and Tia Pilaroscia, and for helpful conversations with Matthias Franz, Ruth Haas, and Hal Schenck.  We are indebted to an anonymous referee for very useful comments.
} \thanks{This work was partially supported by NSF grant DMS--0611020 through the Center for Women in Mathematics at Smith College.  Julianna Tymoczko was partially supported by a Sloan Fellowship and by NSF grants DMS-1101170 and DMS-0801554.}
\author{Simcha Gilbert} 
\address{First and third authors: Department of Mathematics and Statistics, Smith College, Northampton, MA 01063 U.S.A.}
\email{sgilbert@smith.edu}

\author{Shira Polster}
\address{Second author: Department of Mathematics, 2108 SAS Hall, North Carolina State University, Raleigh, NC  27695 U.S.A.}
\email{spolste@ncsu.edu}

\author{Julianna Tymoczko} 
\email{jtymoczko@smith.edu}

\begin{abstract}
Let $G$ be a graph whose edges are labeled by ideals of a commutative ring.  We introduce a {\em generalized spline}, which is a vertex-labeling of $G$ by elements of the ring so that the difference between the labels of any two adjacent vertices lies in the corresponding edge ideal. Generalized splines arise naturally in combinatorics ({\em algebraic splines} of Billera and others) and in algebraic topology (certain equivariant cohomology rings, described by Goresky-Kottwitz-MacPherson and others). The central question of this manuscript asks when an arbitrary edge--labeled graph has nontrivial generalized splines.  The answer is `always', and we prove the stronger result that 
the module of generalized splines contain a free submodule whose rank is the number of vertices in $G$.  We describe the module of generalized splines when $G$ is a tree, and give several ways to describe the ring of generalized splines as an intersection of generalized splines for simpler subgraphs of $G$.  We also present a new tool which we call the {\em GKM matrix}, an analogue of the incidence matrix of a graph, and end with open questions.
\end{abstract}

\maketitle


 
 \section{Introduction}

The goal of this paper is to generalize and extend combinatorial constructions that have become increasingly important in many areas of algebraic geometry and topology, as well as to establish a firm combinatorial footing for these constructions.  Given a commutative ring $R$ with identity, an arbitrary graph $G=(V,E)$, and a function $\alpha: E \rightarrow \{\textup{ideals }I \subseteq R\}$,  we will define a ring of {\em generalized splines}.  This paper does four things:
\begin{enumerate}
\item proves foundational results about generalized splines;
\item completely analyzes the ring of generalized splines for trees and shows families of generalized splines for arbitrary cycles; 
\item produces an $R$-submodule within the ring of generalized splines that has rank $|V|$, as long as $R$ is an integral domain; and 
\item shows that the study of generalized splines for arbitrary graphs can be reduced to the case of different subgraphs, especially cycles or trees. 
\end{enumerate}

Generalized splines as we define them are a subring of a product of copies of $R$:

\begin{definition}
The ring of generalized splines $R_{G}$ of the pair $(G, \alpha)$ is defined by
\[R_{G} = \{\spl{p}{} \in R^{|V|}: \textup{ for each edge } e = uv, \textup{ the difference }\spl{p}{u} - \spl{p}{v} \in \alpha(e)\}.\]
\end{definition}
Figures \ref{fig:K4 spline} and \ref{fig:nonsplines} display examples and nonexamples of elements of $R_{K_4}$ in the case when each ideal $\alpha(e)$ is generated by a single ring element (given inside $\langle \cdots \rangle$).  The vertices are labeled with elements of $R_{K_4}$ and the collection of vertex--labels in Figure \ref{fig:K4 spline} is a  generalized spline. Note that Figures \ref{fig:P_4} and \ref{fig:C_4} are not generalized splines on $K_4$ but {\em are} generalized splines on the subgraphs in bold.  These examples hold for any ring $R$ and any choice of elements $\alpha_1, \alpha_2,\ldots, \alpha_6 \in R$ to generate the ideals $\alpha(e)$.

The name {\em generalized spline} comes from one of the important constructions that we extend.  Historically, engineers modeled complicated objects like ships or cars by identifying important points of the vehicle and then attaching thin strips of wood (called {\em splines}) at those points to approximate the entire hull.  

Mathematically, a spline is a collection of polynomials on the faces of a polyhedral complex that agree (modulo a power of  a linear form) on the intersections of two faces. We refer to this classical tradition as the {\em analytic} approach to splines; it studies the vector space $C^r_k(D)$ where $D$ is a simplicial complex, $r$ is the order of smoothness to which the polynomials agree over faces, and $k$ is the maximal degree of a polynomial supported on a maximal face.  Splines are used in approximation theory and numerical analysis, with applications in data interpolation, to create smooth curves in computer graphics, to find numerical solutions to partial differential equations, and for other applications \cite{Bar84, CL1, CL2}.  

  In the analytic tradition, mathematicians seek individual splines satisfying particular properties as well as characterizations of the space of splines associated to a given object---for instance, the dimension \cite{Al1, Al2, AS1, AS2, CH, GP, Schu1, Schu2, Schu3} or basis \cite{APS2, APS1, MorS, Schu4} for a space of splines.  Alfeld-Schumaker's work is both representative and epitomic: a seminal result of theirs proved a bound on the dimension of $C^r_k(D)$ when $D$ is a planar simplicial complex and $k \geq 3r+1$  \cite{AS2}.

Billera pioneered the study of what some call {\em algebraic} splines, introducing methods from homological and commutative algebra to prove a conjecture of Strang on the dimension of $C^1_k(D)$ when $D$ is a generic planar simplicial complex \cite{B}.  In the abstract algebraic setting, mathematicians generalize the class of geometric objects associated to splines (e.g. work of Schumaker, Billera and Rose, or McDonald and Schenck that studies piecewise polynomials on a polyhedral complex rather than just a simplicial complex  \cite{Schu2, BR1, MS}) and study algebraic invariants of modules other than dimension or bases (e.g.  the more fundamental question of freeness \cite{H, BR2, Y, DS,  Di}, or more algebraically involved questions like  computing coefficients of the Hilbert polynomial \cite{BR1, SS1, Sch1, M, MS}, identifying the syzygy module of the span of the edge ideals \cite{Schu1, R1, R2}, or analyzing algebraic varieties associated to the piecewise polynomials \cite{Wa, ZW1, ZW2}). Billera and Rose introduced a description of splines in terms of the dual graph of the polyhedral complex that is equivalent to the piecewise-polynomial definition for so-called {\em hereditary} complexes \cite{BR1}.  Many others used Billera and Rose's approach in later research \cite{MS, R1, R2}, and it is our starting point.

In what we might call the {\em topological} approach to splines, geometers and topologists recently and independently rediscovered splines as equivariant cohomology rings of toric and other algebraic varieties (though they rarely use the name `splines') \cite{Bri96, P, BFR, Sch2}.  Goresky--Kottwitz--MacPherson developed a combinatorial construction of equivariant cohomology called {\em GKM theory} \cite{GKM}, which can be used for any algebraic variety $X$ with an appropriate torus action.  Unknowingly, they described precisely the dual--graph construction of splines: GKM theory builds a graph $G_X$ whose vertices are the $T$-fixed points of $X$ and whose edges are the one-dimensional orbits of $X$.  Each edge of this graph is associated to a principle ideal $\langle \alpha_e \rangle$ in a polynomial ring, coming from the weight $\alpha_e$ of the torus action on the one-dimensional torus orbits in $X$.  The GKM ring associated to the pair $(G_X, \alpha)$ agrees with what we call the ring of generalized splines for $(G_X, \alpha)$.  The main theorem of GKM theory asserts that under appropriate conditions, this GKM ring is in fact isomorphic to the equivariant cohomology ring $H^*_T(X; \mathbb{C})$.  (Their work relies on earlier work of many others, including a much more general result of Chang-Skjelbred that points to one way to extend this work topologically to cases in which the ideals $\langle \alpha_e \rangle$ are no longer principal \cite{CS1}.) We omit details of the topological background here because there are several excellent surveys \cite{KT, T1, Ho}.  However, GKM theory is a powerful tool in Schubert calculus \cite{GT, KT}, symplectic geometry \cite{GT, GHZ, HH},  representation theory \cite{F}, and other fields.  (In some of these applications, the {\em ring} structure of splines is more important than the module structure.)

We note that the most powerful results in each of these approaches are not replicable using other approaches.  For instance Mourrain and Villamizar recently used the algebraic approach to try to reprove Alfeld and Schumaker's results, but could not attain their bound \cite{MV}.

\begin{figure*}[h]
\begin{picture}(150,110)(-20,0)
\small
\put(0,10){\circle*{5}}
\put(0,100){\circle*{5}}
\put(90,10){\circle*{5}}
\put(90, 100){\circle*{5}}
\put(0,10){\line(0,1){90}}
\put(0,100){\line(1,-1){90}}
\put(0,10){\line(1,0){90}}
\put(90,10){\line(0,1){90}}
\put(0,100){\line(1,0){90}}
\put(0,10){\line(1,1){90}}
\put(35,105){\textcolor{gray}{$\langle \alpha_1\rangle$}}
\put(92,50){\textcolor{gray}{$\langle \alpha_2\rangle$}}
\put(35,-2){\textcolor{gray}{$\langle \alpha_3\rangle$}}
\put(-23,50){\textcolor{gray}{$\langle \alpha_4\rangle$}}
\put(-53,80){\textcolor{gray}{$\langle \alpha_5\rangle$}}
\put(-30,83){\color{gray}\vector(1,0){42}}
\put(117,80){\textcolor{gray}{$\langle \alpha_6\rangle$}}
\put(115,83){\color{gray}\vector(-1,0){40}}
\put(-10,100){\textcolor{red}{$0$}}
\put(95,100){\textcolor{red}{$\alpha_1\alpha_4\alpha_5\alpha_6$}}
\put(95,8){\textcolor{red}{$\alpha_1\alpha_4\alpha_5\alpha_6+\alpha_2\alpha_4\alpha_5\alpha_6$}}
\put(-113,8){\textcolor{red}{$\alpha_1\alpha_4\alpha_5\alpha_6+\alpha_2\alpha_4\alpha_5\alpha_6$}}
\put(-63,-4){\textcolor{red}{$+\alpha_3\alpha_4\alpha_5\alpha_6$}}
\end{picture}  
       \caption{Example of a generalized spline on $K_4$}
       \label{fig:K4 spline}          
\end{figure*}
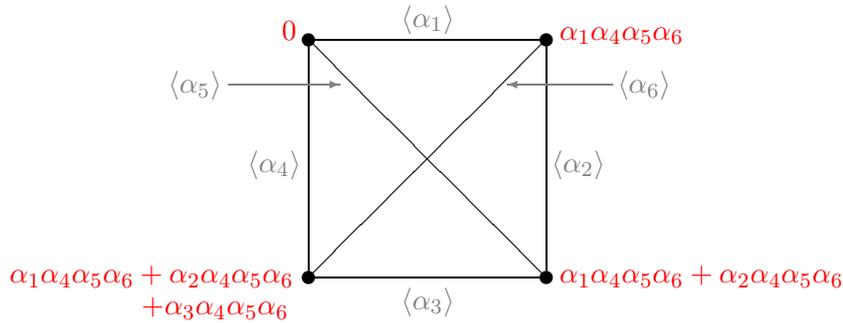

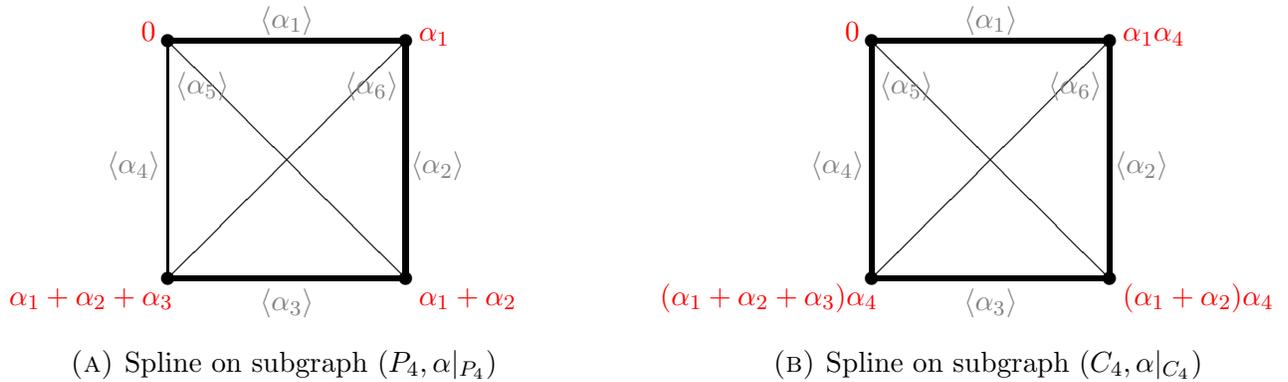
\begin{figure*}[h]      
   \begin{subfigure}[b]{0.4\textwidth}
   \small
\begin{picture}(100,110)(-50,-10)
\put(0,10){\circle*{5}}
\put(0,100){\circle*{5}}
\put(90,10){\circle*{5}}
\put(90, 100){\circle*{5}}
\put(0,10){\line(0,1){90}}
\linethickness{.18em}
\put(0,100){\line(1,-1){90}}
\put(0,10){\line(1,0){90}}
\put(90,10){\line(0,1){90}}
\put(0,100){\line(1,0){90}}
\put(0,10){\line(1,1){90}}
\put(35,105){\textcolor{gray}{$\langle \alpha_1\rangle$}}
\put(92,50){\textcolor{gray}{$\langle \alpha_2\rangle$}}
\put(35,-2){\textcolor{gray}{$\langle \alpha_3\rangle$}}
\put(-23,50){\textcolor{gray}{$\langle \alpha_4\rangle$}}
\put(3,80){\textcolor{gray}{$\langle \alpha_5\rangle$}}
\put(67,80){\textcolor{gray}{$\langle \alpha_6\rangle$}}
\put(-10,100){\textcolor{red}{$0$}}
\put(95,100){\textcolor{red}{$\alpha_1$}}
\put(95,0){\textcolor{red}{$\alpha_1+\alpha_2$}}
\put(-60,0){\textcolor{red}{$\alpha_1+\alpha_2+\alpha_3$}}
\end{picture}
                \caption{Spline on subgraph $(P_4, \alpha|_{P_4})$}
                \label{fig:P_4}
        \end{subfigure}
          \hspace{6em}
        \begin{subfigure}[b]{0.4\textwidth}
        \small
\begin{picture}(150,110)(-50,-10)
\put(0,10){\circle*{5}}
\put(0,100){\circle*{5}}
\put(90,10){\circle*{5}}
\put(90, 100){\circle*{5}}
\linethickness{.18em}
\put(0,10){\line(0,1){90}}
\put(0,100){\line(1,-1){90}}
\put(0,10){\line(1,0){90}}
\put(90,10){\line(0,1){90}}
\put(0,100){\line(1,0){90}}
\put(0,10){\line(1,1){90}}
\put(35,105){\textcolor{gray}{$\langle \alpha_1\rangle$}}
\put(92,50){\textcolor{gray}{$\langle \alpha_2\rangle$}}
\put(35,-2){\textcolor{gray}{$\langle \alpha_3\rangle$}}
\put(-23,50){\textcolor{gray}{$\langle \alpha_4\rangle$}}
\put(3,80){\textcolor{gray}{$\langle \alpha_5\rangle$}}
\put(67,80){\textcolor{gray}{$\langle \alpha_6\rangle$}}
\put(-10,100){\textcolor{red}{$0$}}
\put(95,100){\textcolor{red}{$\alpha_1\alpha_4$}}
\put(95,0){\textcolor{red}{$(\alpha_1+\alpha_2)\alpha_4$}}
\put(-80,0){\textcolor{red}{$(\alpha_1+\alpha_2+\alpha_3)\alpha_4$}}
\end{picture}
               \caption{Spline on subgraph $(C_4, \alpha|_{C_4})$}
                \label{fig:C_4}
        \end{subfigure} 
       \caption{Nonexamples of  generalized splines on $K_4$}\label{fig:nonsplines}          
\end{figure*}

Our definition of generalized splines allows us to do several things that weren't possible from the algebraic or geometric perspectives:
\begin{itemize}
\item {\em We give a lower bound for one of the central questions of classical splines.}  Corollary \ref{cor: free submodule} proves that every collection of generalized splines over an integral domain has a free-submodule of rank $|V|$, producing a lower bound for the dimension of the ring of splines $R_G$ whenever $R_G$ is a free module over $R$.  This significantly generalizes work of Guillemin and Zara in the GKM context \cite[Theorem 2.1]{GuZ2}. 
\item {\em We streamline earlier combinatorial constructions of splines.}  Our construction isolates and highlights the algebraic structures used in previous work on splines.  In our language, algebraic splines assume that the ideals $\alpha(e)$ are principal and that the generators for the ideals $\alpha(e)$ satisfy some coprimality conditions.  A  classical condition like ``piecewise polynomials meet with order $r$ smoothness at an edge $e$" corresponds to using the edge ideal $\alpha(e)^{r+1}$ instead of $\alpha(e)$.  

From the geometric point of view, we owe much to a series of papers by Guillemin and Zara \cite{GuZ1, GuZ2} whose goal is to construct geometric properties of GKM manifolds from a strictly combinatorial viewpoint.  Yet their combinatorial model imposes more restrictions than the classical definition of splines---conditions  that are natural (and necessary!) for any geometric application.
\item {\em We expand the family of objects on which splines are defined to arbitrary graphs.}  Our work shows that graphs that have no reasonable geometric interpretation nonetheless are central to the analysis of splines.  Theorem \ref{theorem: union of graphs} decomposes the ring of splines for a graph $G$ in terms of the splines for subgraphs of $G$; Corollary \ref{cor: SpanningTrees} specializes Theorem \ref{theorem: union of graphs} to spanning trees, whose splines are completely described in Theorem \ref{Tree}; and Theorem \ref{theorem: Intersection} decomposes the ring of splines for $G$ in terms of a particular collection of subcycles and subtrees of $G$.  Cycles play a similarly key role in Rose's description of the syzygies of spline ideals \cite{R1, R2} (see also Schumaker's work \cite{Schu1}).  Yet neither trees nor cycles are geometrically meaningful from a GKM picture.  (See forthcoming work of Handschy, Melnick, and Reinders \cite{HMR} and of Bowden, Cao, Hagen, King, and Reinders \cite{BHKR} for a deeper investigation of generalized splines on cycles.)
\item {\em We expand the family of rings on which splines are defined.}  This give a convenient language to describe simultaneously the GKM constructions for equivariant cohomology and equivariant K--theory.  
Moreover, generalized splines over integers have interesting connections to elementary number theory \cite{HMR}.
\item {\em We provide the natural language for further generalizations of splines.}   Our construction of generalized splines extends even more: label each vertex of the graph $G$ by a (possibly distinct) $R$-module $M_v$ and label each edge by a module $M_e$ together with homomorphisms $M_v \rightarrow M_e$ for each vertex $v$ incident to the edge $e$.  Geometrically, this corresponds to Braden and MacPherson's construction of equivariant intersection homology \cite{BM}, also used by Fiebig in representation-theoretic contexts \cite{F}.  We discuss this and other open questions in the final section of the paper.
\end{itemize}

The rest of this paper is structured as follows.  Section \ref{section: foundational} establishes  essential results for generalized splines that were first shown in special cases like equivariant cohomology and algebraic splines.  We highlight Theorem \ref{thm: module direct sum} and Corollary \ref{cor: syzgies}, which generalize and strengthen Rose's result that for certain polyhedral complexes, the syzygies $B$ of the spline ideal are a direct summand of the splines $R_G \cong R \oplus B$ \cite{R1}.  Corollary \ref{cor: syzgies} uses this in Rose's special case to show that  the syzygies of the ideal generated by the image of $\alpha$ is isomorphic as a module to the collection of generalized splines whose restriction to a particular fixed point is zero.  This relates the algebraically--natural question of finding syzygies of splines to the question of finding a particular, geometrically--natural kind of basis for the module of splines.  Section \ref{GKM Matrix} describes a tool analogous to the incidence matrix of a graph that we call a {\em GKM matrix}.  Section \ref{sec: Trees} completely characterizes the generalized splines for trees  in terms of a minimal set of free generators for the ring of generalized splines.

One of our central questions is: when does an edge--labeled graph have nontrivial generalized splines? The answer (essentially always, as in Theorem \ref{thm: Complete2}) is actually more refined.  Corollary \ref{cor: free submodule} explicitly constructs a free $R$-submodule of the generalized splines on $G$ of rank $|V|$.  When $R$ is an integral domain and the generalized splines form a free $R$-module (as is the case for GKM theory), we conclude that the rank of the $R$-module of generalized splines is at least $|V|$.  

Section \ref{section: submodules} uses analogues of a shelling order (in combinatorics) or a `flow-up basis' (in geometry) to identify $R$-submodules of the generalized splines.  Section \ref{section: intersection} characterizes generalized splines differently: in terms of the intersections of the generalized splines formed by various subgraphs.  This allows us to reframe the definition of generalized splines as an intersection of very simple graphs (Theorem \ref{theorem: union of graphs}) and to reduce the number of intersections needed by using certain spanning trees (Corollary \ref{cor: SpanningTrees}).  Finally, Theorem \ref{theorem: Intersection} analyzes the GKM matrix directly to decompose the ring of generalized splines on $G$ as an intersection of the generalized splines for particular subcycles of $G$.




\section{Definitions and foundational results}\label{section: foundational}

In this section, we formalize a collection of definitions which were stated implicitly in the introduction.   We then give foundational results describing the structure of the ring of generalized splines, including key methods to construct the ring and to build new generalized splines from existing ones. 

We begin with a quick overview of our notational conventions.

\begin{enumerate}
\item $G$: a graph, defined as a set of vertices $V$ and edges $E$. Assumed throughout to be finite with no multiple edges between vertices.
\item $R$: a commutative ring with identity $1$.
\item $\mathcal{I}$: the set of ideals in $R$.
\item $\alpha$: an edge--labeling function on $G$ that assigns a nonzero element of $\mathcal{I}$ to each edge in $E$. See Definition \ref{Edge Label Defn}.
\item $(G,\alpha)$: an edge--labeled graph.
\item $\alpha(e_{i,j})= \alpha(v_iv_j) = I_{e_{i,j}}$: image of the edge $e_{i,j} = v_iv_j$ under the map $\alpha$.
\item $\alpha_{i,j}$: an arbitrary element of the ideal $\alpha(e_{i,j})$. When $\alpha(e_{i,j})$ is principal, $\alpha_{i,j}$ often denotes the generator.
\item $R_{G}$: the ring of generalized splines on $(G, \alpha)$. See Definition \ref{Gen Spline Defn}.
\item $\spl{p}{}$: a generalized spline. An element of $\bigoplus_{v \in V} R$ denoted $\spl{p}{} = (\spl{p}{v_1}, \spl{p}{v_2}, \ldots, \spl{p}{v_{|V|}})$. See Definition \ref{Gen Spline Defn}.
\item $\spl{p}{v}$: the coordinate of $\spl{p}{}$ corresponding to vertex $v \in V$. An element of $R$.
\item $M_{G}$: the (possibly extended) GKM matrix for the graph $G$. See Definition \ref{GKM Matrix Defn}.
\end{enumerate}

The first definition describes the combinatorial set-up of our work: a graph whose edges are labeled by ideals in a ring $R$.  The ring $R$ is always assumed to be a commutative ring with identity, though in later sections we occasionally add more conditions.

\begin{definition}
\label{Edge Label Defn}
Let $G = (V,E)$ be a graph and let $R$ be a commutative ring with identity.  An {\em edge--labeling function} is a map $\alpha: E \rightarrow \{\textup{ideals }I \subseteq R\}$ from the set of edges of $G$ to the set of nonzero ideals in $R$.  An {\em edge--labeled graph} is a pair $(G, \alpha)$ of a graph $G$ together with an edge--labeling function of $E$.  We often refer to the set of ideals in $R$ as $\mathcal{I}$.
\end{definition}

We now  precisely define the compatibility condition that we use on the edges.

\begin{definition}
\label{GKM Condition Defn}
Let $G = (V,\alpha)$ be an edge--labeled graph.  An element $\spl{p}{} \in \bigoplus_{v \in V} R$ satisfies  {\em the GKM condition} at an edge $e=uv$ if $\spl{p}{u} - \spl{p}{v} \in \alpha(e)$.
\end{definition}

In GKM theory and in the theory of algebraic splines, the ring $R$ is a polynomial ring in $n$ variables.  The ideal $\alpha(e)$ is  the principal ideal generated by a linear form in GKM theory, and by a power of a linear form in the theory of algebraic splines.

We build the ring of generalized splines by imposing the GKM condition at every edge in the graph.

\begin{definition}
\label{Gen Spline Defn}
Let  $(G,\alpha)$ be an edge--labeled graph.  The {\em ring of generalized splines} is
\[R_{G, \alpha} = \left\{ \spl{p}{} \in \bigoplus_{v \in V} R \textup{ such that } \spl{p}{} \textup{ satisfies the GKM condition at each edge } e \in E\right\}.\]
 Each element of $R_{G, \alpha}$ is called a {\em generalized spline}.  When there is no risk of confusion, we write $R_{G}$.
\end{definition}

We first confirm that in fact $R_{G}$ is a ring.

\begin{proposition}\label{ring proof}
$R_{G}$ is a ring with unit $\spl{1}{}$ defined by $\spl{1}{v} = 1$ for each vertex $v \in V$.  
\end{proposition}

\begin{proof}
By definition $R_{G}$ is a subset of the product ring $\bigoplus_{v \in V} R$ so we need only confirm that the identity is in $R_{G}$ and that $R_{G}$ is closed under addition and multiplication.  The operations are component-wise addition and multiplication since $R_{G}$ is in $\bigoplus_{v \in V} R$.  The identity in $\bigoplus_{v \in V} R$ is the generalized spline $\spl{1}{}$ defined by $\spl{1}{v} = 1$ for each vertex $v \in V$.  This satisfies the GKM condition at each edge because for every edge $e = uv$ we have $\spl{1}{u} - \spl{1}{v} = 0$ and $0$ is in each ideal $\alpha(e)$.  The set $R_{G}$ is closed under addition because if $\spl{p}{}, \spl{q}{} \in R_{G}$ then for each edge $e=uv$ we have
\[\spl{(p+q)}{u} - \spl{(p+q)}{v} = (\spl{p}{u} + \spl{q}{u}) - (\spl{p}{v} + \spl{q}{v}) = (\spl{p}{u} - \spl{p}{v}) + (\spl{q}{u} - \spl{q}{v}) \]
which is in $\alpha(e)$ by the GKM condition. 
Similarly, the set $R_{G}$ is closed under multiplication because if $\spl{p}{}, \spl{q}{} \in R_{G}$ then for each edge $e=uv$ we have
\[\spl{(pq)}{u} - \spl{(pq)}{v} = (\spl{p}{u} \spl{q}{u}) - (\spl{p}{v} \spl{q}{v}) = (\spl{p}{u} \spl{q}{u}- \spl{p}{v}\spl{q}{u}) + (\spl{p}{v}\spl{q}{u} - \spl{p}{v}\spl{q}{v}) = \spl{q}{u}(\spl{p}{u} - \spl{p}{v}) + \spl{p}{v}(\spl{q}{u} - \spl{q}{v}) \]
which is in $\alpha(e)$ by the GKM condition. 
\end{proof}

The generalized splines $R_G$ also form an $R$--module: multiplication by $r$ corresponds to  scaling each polynomial in the spline $\spl{p}{}$ or equivalently to multiplication by $r \spl{1}{}$.  Figure \ref{fig:P_4 Multiple} demonstrates the $R$-module structure of $R_{P_4}$: multiplying $\spl{p}{}$ by an arbitrary element $r \in R$ produces the spline $r\spl{p}{} = (r\spl{p}{v_1}, r\spl{p}{v_2}, r\spl{p}{v_3}) \in R_{P_4}$.

One major question we study in this paper is whether there are {\em nontrivial} generalized splines, which means the following.

\begin{definition}
A {\em nontrivial generalized spline} is an element $\spl{p}{} \in R_{G}$ that is not in the principal ideal $R\spl{1}{}$.
\end{definition}

In other words, we ask whether the $R$-module $R_{G}$ contains at least two linearly independent elements.  We answer this question completely (and more strongly) in Theorem \ref{thm: Complete2} and its Corollary \ref{cor: Proper Subgraph}: yes, except in the trivial cases when $G$ consists of a single point or $R$ is  zero.

If edge--labels were zero then the ring of splines could be trivial for trivial algebraic reasons: for instance, if all edge--labels of $G$ were zero then the only elements of $R_G$ are trivial splines.  This is why $\alpha(e)$ is always nonzero in Definition \ref{Edge Label Defn}. 

\begin{definition}
Let $(G,\alpha)$ and $(G', \alpha')$ be edge--labeled graphs with respect to $R$.  A {\em homomorphism of edge--labeled graphs} $\phi: (G, \alpha) \rightarrow (G', \alpha')$ is a graph homomorphism $\phi_1: G \rightarrow G'$ together with a ring {\em automorphism} $\phi_2: R \rightarrow R$ so that for each edge $e \in E_{G}$ we have $\phi_2(\alpha(e)) = \alpha'(\phi_1(e))$.
\begin{equation}\label{equation: commutative diagram}
 \begin{CD}
E_{G} @>\phi_1 >> E_{G'} \\
@VV\alpha V @VV \alpha' V\\
\mathcal{I} @> \phi_2 >> \mathcal{I}
\end{CD}
\end{equation}
An {\em isomorphism of edge--labeled graphs} is a homomorphism of edge--labeled graphs whose underlying graph homomorphism is in fact an isomorphism.
\end{definition}

We stress that the map $\phi_2$ is a ring {\em automorphism}.  This ensures that $\phi_2$ induces a map on the set of ideals $\phi_2: \mathcal{I} \rightarrow \mathcal{I}$ and that the diagram in Equation \eqref{equation: commutative diagram} is well-defined.  The content of the definition is that the diagram commutes.

Many interesting homomorphisms of edge--labeled graphs arise when $\phi_2: R \rightarrow R$ is the identity homomorphism.  Indeed, when $R$ is the integers, this is essentially the only case.  However, some rings $R$ have very interesting automorphisms: for instance, when $R$ is a polynomial on $n$ variables, the symmetric group on $n$ letters acts on $R$ by permuting variables.  This induces an important action in equivariant cohomology, which is substantively different from a closely related action induced by the identity ring automorphism \cite{T2}. Our first proposition confirms that the ring of generalized splines is invariant under edge--labeled isomorphisms.  More precisely, when two graphs are edge--labeled isomorphic, any generalized spline for one graph will be a generalized spline for the other.  

\begin{proposition}
\label{proposition: isomorphism}
If $\phi: (G, \alpha) \rightarrow (G', \alpha')$ is an isomorphism of edge--labeled graphs then $\phi$ induces an isomorphism of the corresponding rings of generalized splines $\phi_*: R_{G} \cong R_{G'}$ according to the rule that $\phi_*(\spl{p}{})_{\phi_1(u)} = \phi_2(\spl{p}{u})$ for each $u \in V_{G}$.
\end{proposition}


\begin{proof}
By definition of generalized splines
\[\spl{p}{} \in R_{G} \textup{  if and only if  } \spl{p}{u} - \spl{p}{v} \in \alpha(e)
\textup{  for each edge  }e=uv \textup{ in }E_{G}.\] 
The map $\phi_2: R \rightarrow R$ is an automorphism of rings, so the GKM conditions imply
\begin{equation} \label{equation: first automorphism line}
\spl{p}{} \in R_{G} \textup{  if and only if  } \phi_2(\spl{p}{u}) - \phi_2(\spl{p}{v}) \in \phi_2(\alpha(e))
\textup{  for each edge  }e=uv \textup{ in }E_{G}.
\end{equation}
The map $\phi_1$ is an isomorphism between the underlying graphs $G$ and $G'$, so $e$ is an edge in $G$ if and only if $\phi_1(e)$ is an edge in $G'$. With the fact that $\alpha'(\phi_1(e)) = \phi_2(\alpha(e))$ for each edge $e \in E_{G}$, this means Equation \eqref{equation: first automorphism line} is equivalent to:
\begin{equation}
\label{equation: preserved under isomorphism}
\phi_2(\spl{p}{u}) - \phi_2(\spl{p}{v}) \in \alpha'(\phi_1(e)) \textup{  for each edge  } e' = \phi_1(u)\phi_1(v) \textup{  in  } E_{G'}.
\end{equation}
Equation \eqref{equation: preserved under isomorphism} 
is equivalent to $\phi_*(\spl{p}{}) \in R_{G'}$ so we conclude that $\spl{p}{}$ is a generalized spline in $R_{G}$ if and only if $\phi_*(\spl{p}{})$ is in $R_{G'}$.
\end{proof}

The next proposition verifies that a generalized spline for the pair $(G, \alpha)$ is a generalized spline for every subgraph of $G$.

\begin{proposition}\label{proposition: subgraphs}
Let $(G, \alpha)$ be an edge--labeled graph and $G' = (V^\prime, E^\prime)$ a subgraph of $G$.  Let $(G', \alpha|_{E^\prime})$ be the edge--labeled graph
whose function $\alpha|_{E^\prime}$ denotes the restriction of $\alpha$ to the edge set of $G'$. 
If $\spl{p}{}$ is a generalized spline for $(G, \alpha)$ then $\spl{p}{}|_{V'} \in \bigoplus_{v \in V^\prime}R $ is a generalized spline for $(G', \alpha|_{E'})$.\end{proposition}

\begin{proof}
Let $G' \subseteq G$ as in the hypothesis, let $\spl{p}{}$ be a generalized spline for $(G,\alpha)$, and consider the subcollection 
$\spl{p}{}|_{V'}$ obtained by restricting $\spl{p}{}$ to the vertex set $V^\prime \subseteq V$ of $G^\prime$.  For any edge $v_iv_j$ in $G'$ the corresponding edge $v_iv_j \in  E$ since $E' \subseteq E$.  This implies that $\spl{p}{v_i} - \spl{p}{v_j} \in \alpha(v_iv_j)$  by the GKM condition for $(G, \alpha)$.  Since the edge--labeling function for $G'$ is the restriction $\alpha|_{E^\prime}$ to the edges in $E^\prime \subseteq E$, we conclude that the GKM condition is satisfied at every edge of $G'$.  It follows that $\spl{p}{}|_{V'}$ is a generalized spline for $(G', \alpha|_{E'})$ as desired.
\end{proof}

\begin{example}
Consider the generalized spline on the bold $P_4$ in Figure \ref{fig:P_4} with edges labeled as in in Figure \ref{fig:K4 spline}. Removing a leaf and its incident edge from $P_4$ gives the subgraph $P_3$ in Figure \ref{fig:P_3}. The generalized spline for $P_4$ still satisfies the GKM condition at every vertex on the subgraph. Thus $\spl{p}{}|_{P_3}$ is a generalized spline for $P_3$. 
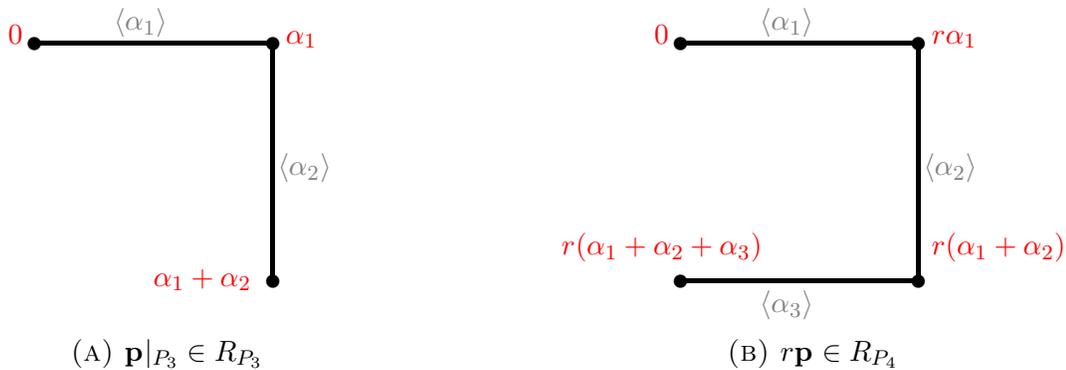
\begin{figure}[h] 
   \begin{subfigure}[b]{0.4\textwidth}
   \centering     
   \small
\begin{picture}(100,110)(0,-5)
\put(0,100){\circle*{5}}
\put(90,10){\circle*{5}}
\put(90, 100){\circle*{5}}
\linethickness{.13em}
\put(90,10){\line(0,1){90}}
\put(0,100){\line(1,0){90}}
\put(30,105){\textcolor{gray}{$\langle \alpha_1 \rangle$}}
\put(92,50){\textcolor{gray}{$\langle \alpha_2 \rangle$}}
\put(-10,100){\textcolor{red}{$0$}}
\put(95,100){\textcolor{red}{$\alpha_1$}}
\put(45,8){\textcolor{red}{$\alpha_1+\alpha_2$}}
\end{picture}
                \caption{$\spl{p}{}|_{P_3} \in R_{P_3}$}
                \label{fig:P_3}
        \end{subfigure}
          \hspace{4em}
        \begin{subfigure}[b]{0.4\textwidth}
        \centering
        \small
\begin{picture}(100,110)(0,-5)
\put(0,10){\circle*{5}}
\put(0,100){\circle*{5}}
\put(90,10){\circle*{5}}
\put(90, 100){\circle*{5}}
\linethickness{.13em}
\put(0,10){\line(1,0){90}}
\put(90,10){\line(0,1){90}}
\put(0,100){\line(1,0){90}}
\put(30,105){\textcolor{gray}{$\langle \alpha_1 \rangle$}}
\put(92,50){\textcolor{gray}{$\langle \alpha_2 \rangle$}}
\put(30,-2){\textcolor{gray}{$\langle \alpha_3 \rangle$}}
\put(-10,100){\textcolor{red}{$0$}}
\put(95,100){\textcolor{red}{$r\alpha_1$}}
\put(95,20){\textcolor{red}{$r(\alpha_1+\alpha_2)$}}
\put(-45,20){\textcolor{red}{$r(\alpha_1+\alpha_2+\alpha_3)$}}
\end{picture}
                \caption{$r\spl{p}{} \in R_{P_4}$}
                \label{fig:P_4 Multiple}
\end{subfigure}
\caption{New generalized splines from old}\label{figure: new splines from old}
\end{figure}
\end{example}

The next proposition shows that the special case when one of the edges is associated to the unit ideal $\alpha(e) = R$ is equivalent to a kind of restriction as in Proposition \ref{proposition: subgraphs}.  In this case, the edge $e$ can be erased  without affecting the ring of generalized splines.

\begin{proposition}\label{proposition: edge deletion}
Suppose that the edge--labeled graph $(G, \alpha)$ has an edge $e$ with $\alpha(e)=R$.  Let $G' = (V_{G}, E-\{e\})$ be the graph $G$ with edge $e$ erased, and let $\alpha': E - \{e\} \rightarrow \mathcal{I}$ be the restriction $\alpha' = \alpha|_{E - \{e\}}$.  Then
\[R_{G} = R_{G'}\]
\end{proposition}

\begin{proof}
Proposition \ref{proposition: subgraphs} says that every generalized spline of $G$ is a generalized spline of $G'$, since $G'$ is a subgraph of $G$ with the same vertex set whose labeling agrees on shared edges.  Hence $R_{G} \subseteq R_{G'}$.  To prove the converse, suppose $\spl{p}{}$ is a generalized spline for $(G', \alpha')$.  The GKM condition guarantees that $\spl{p}{u} - \spl{p}{v} \in \alpha(uv)$ 
for every edge $uv \in E-\{e\}$.  In addition, if $u_0, v_0$ are the endpoints of the edge $e$, then $\spl{p}{u_0} - \spl{p}{v_0} \in R$ is vacuously true.  Since $\alpha(e)=R$ we conclude that the GKM condition is satisfied for the edge $e$ as well.  So $\spl{p}{} \in R_{G}$ and $R_{G'} = R_{G}$ as desired.
\end{proof}

We may  build generalized splines from disjoint unions of graphs by taking the direct sum of the respective generalized splines.

\begin{proposition}
If $G = G_1 \cup G_2$ is the union of two disjoint graphs then $R_G = R_{G_1} \oplus R_{G_2}$.
\end{proposition}

\begin{proof}
Rearranging the GKM conditions gives the result:
\[\begin{array}{lcl}
R_G &=& \left\{ \spl{p}{} \in \bigoplus_{v \in V} R \textup{ such that } \spl{p}{} \textup{ satisfies the GKM condition at each edge } e \in E(G) \right\} \\
&=& \left\{\spl{p}{} \in \bigoplus_{v \in V(G_1)} R \textup{ such that } \spl{p}{v}-\spl{p}{u} \in \alpha(uv) \textup{ for all } uv \in E(G_1) \right\} \\
&& \hspace{0.5in} \oplus 
     \left\{\spl{p}{} \in \bigoplus_{v \in V(G_2)} R \textup{ such that } \spl{p}{v}-\spl{p}{u} \in \alpha(uv) \textup{ for all } uv \in E(G_2) \right\} \\
     &=& R_{G_1} \oplus R_{G_2} \end{array}\]
     because the vertex sets of $G_1$ and $G_2$ are disjoint. 
\end{proof}

Another approach to constructing generalized splines is to build them one vertex at a time.  The next result decomposes the $R$-module of generalized splines into a direct sum of the trivial generalized splines and the generalized splines that are zero at a particular vertex.

\begin{theorem}
\label{thm: module direct sum}
Suppose that $G$ is a connected graph with edge--labeling function $\alpha: V \rightarrow \mathcal{I}$.  Fix a vertex $v \in V$.  Then every generalized spline $\spl{p}{} \in R_{G}$ can be written uniquely as $\spl{p}{} = r \spl{1}{} + \spl{p^v}{}$, where $\spl{p^v}{}$ is a generalized spline satisfying $\spl{p^v}{v}=0$ and $r \in R$ satisfies $r = \spl{p}{v}$.  In other words, if $M =  \langle \spl{p}{}: \spl{p}{v}=0 \rangle$ then $R_{G} = R \spl{1}{} \oplus M$ as $R$-modules.
\end{theorem}

\begin{proof}
The trivial generalized spline $\spl{1}{}$ is in $R_{G}$ by Proposition \ref{ring proof}.  Let $r \in R$ be the element $r = \spl{p}{v}$.  Then define $\spl{p^v}{}$ to be the generalized spline $\spl{p^v}{} = \spl{p}{} - r \spl{1}{}$.  (There is a unique element in the ring $R_{G}$ that satisfies this equation.)  By construction we have
\[ \spl{p^v}{v} = \spl{p}{v} - r \spl{1}{v} = r - r = 0.\]
This proves the claim.
\end{proof} 

The previous result could lead us to consider $R$-module bases of generalized splines; see the open questions in Section \ref{section: open}.  Instead, we combine it with a result of Rose's to relate the generalized splines that vanish at a particular vertex to the syzygies of the module generated by the edge--ideals.  (Schumaker also implicitly considered syzygies in an earlier work on splines \cite{Schu1}.)

\begin{corollary}
\label{cor: syzgies}
Suppose $G$ is the dual graph of a hereditary polyhedral complex $\Delta$ and that $R$ is the polynomial ring $\mathbb{R}[x_1, x_2, \ldots, x_d]$.  For each edge $e$ in $G$, let $\ell_e$ be an affine form generating the polynomials vanishing on the intersection of faces in $\Delta$ corresponding to $e$.  Define $\alpha$ to be the function $\alpha(e) = \langle \ell_e^{r+1} \rangle$ for each edge $e$ and let
\[B = \left\{ (b_1, \ldots, b_{|E|}) \in R^{|E|}: \textup{ for all cycles $C$ in $G$, the linear combination } \sum_{e \in C} b_e \ell_e^{r+1} = 0  \right\}.\]

Then $M \cong B$ as $R$-modules.
\end{corollary}

\begin{proof}
Under these conditions, Rose proved that $R_G \cong R \oplus B$ as $R$-modules \cite[Theorem 2.2]{R1}.  From the previous claim, we conclude $M \cong B$ as desired.
\end{proof}

We close this section by describing the relationship between the ring of generalized splines associated to an edge--labeling $\alpha$ and the ring of generalized splines associated to the edge--labeling $r\alpha$ obtained by scaling.

\begin{theorem}
Suppose that $(G, \alpha)$ is a connected edge--labeled graph.  Fix an element $r \in R$ and define the edge--labeling function $r\alpha: E \rightarrow \mathcal{I}$  by $r\alpha(e) = rI_e$ for each edge $e \in E$.  Choose a vertex $v_0 \in V$ and define $M =  \langle \spl{p}{}: \spl{p}{v_0}=0 \rangle$.  

If $R$ is an integral domain then
\[R_{G, r\alpha} = R \spl{1}{} \oplus rM.\]
\end{theorem}

\begin{proof}
Theorem \ref{thm: module direct sum} showed that $R_{G, \alpha} = R \spl{1}{} \oplus M$.  The multiple $rR_{G, \alpha} \subseteq R_{G, r \alpha}$ by definition, so $rM \subseteq R_{G, r \alpha}$.  We also know the intersection $rM \cap R \spl{1}{}$ is zero since the only element of $R\spl{1}{}$ whose restriction to $v_0$ vanishes is the zero spline.  So $R_{G, r\alpha} \supseteq R \spl{1}{} \oplus rM$.

We now prove the opposite containment.  Suppose $\spl{p'}{} \in R_{G, r\alpha}$ and let $\spl{p}{} = \spl{p'}{} - \spl{p'}{v_0} \spl{1}{}$.  (Note that $\spl{p}{}$ satisfies the GKM condition for $(G, r\alpha)$ at each edge.)  We will prove that $\spl{p}{} \in rM$.  We split the argument into two pieces: showing that $\spl{p}{}$ is divisible by $r$ {\em at each vertex}, and then showing that $\spl{p}{}$ satisfies the GKM conditions of $rM$.

To begin, we prove by induction that if $v_k$ is connected to $v_0$ by a path of length $k$ then $\spl{p}{v_k} \in rR$ is in the principal ideal generated by $r$.  The unique path of length zero is our base case, and the element $\spl{p}{v_0}=0 \in rR$ by construction.  
Suppose the claim is true for paths of length $k-1$ and let $v_k$ be a vertex connected to $v_0$ by a path of length $k$.  Then $v_k$ is adjacent to a vertex $v_{k-1}$ which is connected to $v_0$ by a path of length $k-1$.  We know $\spl{p}{v_{k-1}} \in rR$ by the inductive hypothesis, and $\spl{p}{v_k} - \spl{p}{v_{k-1}} \in r I_{e_k}$ for the edge $e_k = v_{k-1}v_k$ by the GKM condition.  The sum $r I_{e_k}+rR \subseteq rR$ since ideals are closed under addition, so $\spl{p}{v_k} \in rR$ as desired.  By induction and because $G$ is connected, we conclude that $\spl{p}{v} \in rR$ for all $v \in V$.  

We just showed that each ring element $\spl{p}{}$ is divisible by $r$.  For each vertex $v$, let $\spl{q}{v}$ be the ring element with $\spl{p}{v} = r\spl{q}{v}$ and collect the $\spl{q}{v}$ into the element $\spl{q}{} \in R^{|V|}$.  We ask whether $\spl{q}{} \in M$.  To answer this, we need to know that for each edge $e=uv$ we have $\spl{q}{u} - \spl{q}{v} \in I_e$.  We know that $\spl{p}{u} - \spl{p}{v} \in r I_e$ by the GKM condition.  Let $x = \spl{q}{u} - \spl{q}{v} \in R$ to isolate the underlying algebraic question: if $rx \in r I_e$ then is $x \in I_e$?  The answer is yes when $R$ is an integral domain: if $rx \in r I_e$ then we can find $y \in I_e$ with $rx=ry$.  Hence $r(x-y)=0$, which implies $x=y$ as long as $R$ is an integral domain.  This proves the claim.
\end{proof}

\section{The GKM  matrix}
\label{GKM Matrix}

The results in the previous section allow us to build new generalized splines from existing ones.  To construct generalized splines from scratch we need a systematic method for recording and analyzing GKM conditions. We do this by representing GKM conditions in matrix form. This section shows how to construct GKM matrices and gives several examples.

Our definition of the GKM matrix assumes that the graph $G$ is directed.  Remark \ref{remark: orientations are irrelevant} shows that changing the directions on the edges of $G$ does not affect the solution space of the matrix, so we generally omit orientations from our figures and our discussion.

\begin{definition}
\label{GKM Matrix Defn}
The {\em GKM matrix} of the directed, edge--labeled graph $(G,\alpha)$ is an $|E| \times |V|$ matrix constructed so that the row corresponding to each directed edge $e=uv \in E$ has
\begin{itemize}
\item $1$ in the column corresponding to $u$,
\item $-1$ in the column corresponding to $v$, and
\item zero otherwise.
\end{itemize}
An {\em extended GKM matrix} of the pair $(G,\alpha)$ is an $|E| \times (|V|+1)$ matrix whose first $|V|$ columns are the GKM matrix, and whose last entry in the row corresponding to edge $e$ is any element $\alpha_e \in \alpha(e)$. 

 When there is no risk of confusion, we refer to an extended GKM matrix as simply the GKM matrix.\end{definition}

For instance, if $\alpha(e) = \langle \alpha_{e_1}, \ldots, \alpha_{e_m}\rangle$ is finitely generated, we could write the last entry in the row corresponding to $e$ as $q_{e_1}\alpha_{e_1} + \ldots + q_{e_m}\alpha_{e_m}$ for arbitrary $q_{e_i} \in R$. In particular, if the ideal $\alpha(e)$ is principal and $\alpha(e) = \langle \alpha_e \rangle$ then we typically write the last column of the extended GKM matrix as the vector $(q_e \alpha_e)_{e \in E}$ for arbitrary coefficients $q_e \in R$. 

\begin{remark} \label{remark: syzygies}
Using this language, we can reframe the syzygy module of spline ideals that Rose defined and that we saw in Corollary \ref{cor: syzgies}.  (See also Schumaker's work \cite{Schu1}.)  In our context, the syzygy module is essentially the collection of elements $q_e \in \alpha(e)$ from the edge-ideals so that $\sum_{e \in C}q_e = 0$ for each cycle $C$ in $G$.  In other words, it describes a collection of elements $q_e \in \alpha(e)$ for which the extended GKM matrix represents a homogeneous system of equations.  This condition appears naturally as we analyze the ring $R_G$ further in Theorem \ref{theorem: Intersection}.
\end{remark}

Generally we consider $q_e$ to be a parameter that takes values in $R$, as in the following proposition, which follows immediately from the construction of the GKM matrix.

\begin{proposition}
\label{prop: solution set to GKM}
Let $M_{G}$ denote the GKM matrix of $(G, \alpha)$. Then the spline $\spl{p}{} \in R^{|V|}$ is a generalized spline for $(G, \alpha)$ if and only if there is an extended GKM matrix $[M_{G}|\spl{v}{}]$ for which $\spl{p}{}$ is a solution.
\end{proposition}

\begin{proof}
The matrix $M_{G}$  is constructed to record the GKM condition at every edge $e_{i,j} \in E(G)$.   Hence a spline $\spl{p}{} = \left[ \begin{array}{c} \spl{p}{v_1} \\ \vdots \\ \spl{p}{v_{|V|}} \end{array} \right]  \in R^{|V|}$ is a generalized spline for $(G, \alpha)$  if and only if $M_{G}\spl{p}{} = \spl{v}{}$ for some vector $\spl{v}{}= (\alpha_e)_{e \in E}$.  This is equivalent to saying the spline $\spl{p}{}$ is a solution to the system $[M_{G}|\spl{v}{}]$ for some extended GKM matrix, as claimed.
\end{proof}

We can now manipulate $M_{G}$ to obtain systems of equations that are equivalent to the original GKM conditions on $G$.  We state the following corollary simply to stress this fundamental linear algebra property.  

\begin{corollary}\label{corollary: solution set is gen splines}
If $[M'|\spl{v'}{}]$ is obtained from $[M|\spl{v}{}]$ by a series of reversible row or column operations, then the solution set in $R^{|V|}$ to $[M'|\spl{v'}{}]$ is the same as that of $[M|\spl{v}{}]$.
\end{corollary}

Reversible operations correspond to invertible matrices in $GL_{|V|}(R)$.  For instance, multiplying a row by $x$ is not reversible for the ring $R = \mathbb{C}[x]$ since $1/x$ is not in $R$.  However, multiplying a row by $x$ is reversible when $R = \mathbb{C}(x)$.

\begin{remark}\label{remark: orientations are irrelevant}
Changing the direction of a given edge in $G$ amounts to multiplying the corresponding row in $M_G$ by $-1$, a reversible operation. Hence while the definition of the GKM matrix for the pair $(G, \alpha)$ requires a directed graph, the actual direction chosen is irrelevant to the solution set given by Proposition \ref{prop: solution set to GKM}.
\end{remark}

\begin{example}\label{P3 Example}
We start with the path $P_3$ from Figure \ref{fig:P_3}.  Its extended GKM matrix is
\begin{align*}
M_{P_3}& = \left(\left[ \begin{array}{rrr|l} 1 & -1 &0& q_1\alpha_1 \\ 0 & 1 & -1 & q_2\alpha_2  \end{array} \right]\right) \end{align*}
whose rows may be added to obtain the equivalent system
\[\left[ \begin{array}{rrr|l} 1 & 0 &-1 & q_2\alpha_2 + q_1\alpha_1 \\ 0 & 1 & -1 & q_2\alpha_2  \end{array} \right] \]
If $\spl{p}{} = (\spl{p}{v_1}, \spl{p}{v_2}, \spl{p}{v_3}) \in R_{P_3}$ then the system has dependent variables $\spl{p}{v_1} $ and $\spl{p}{v_2} $ and independent variable $\spl{p}{v_3}$.  All solutions may be written in the form
\begin{align*}
\spl{p}{v_1} &= \spl{p}{v_3} + q_2\alpha_2 + q_1\alpha_1 \\
\spl{p}{v_2} & = \spl{p}{v_3} + q_2\alpha_2
\end{align*}
where $\spl{p}{v_3}$, $q_1$, and $q_2$ are freely chosen elements of $R$.

Setting $\spl{p}{v_3}=0, q_1=1$ and $q_2=1$ yields the generalized spline in Figure \ref{fig:P_3}.

\end{example}

The following generalization will be a central part of our proof of Theorem \ref{thm: Cycle1}.

\begin{example}
Consider the path $P_n$ on $n$ vertices:
\begin{center}
			\begin{picture}(30,30)
				\put(-80,20){\circle*{5}}
				\put(-30,20){\circle*{5}}
				\put(-10,20){\circle*{2}}
				\put(-5,20){\circle*{2}}
				\put(0,20){\circle*{2}}
				\put(20,20){\circle*{5}}
				\put(70,20){\circle*{5}}
				
				\put(-80,20){\line(1,0){50}}
				\put(-30,20){\line(1,0){10}}
				\put(20,20){\line(-1,0){10}}
				\put(20,20){\line(1,0){50}}
				
				\put(-125,18){$P_{n} =$}
				\put(-83,5){$v_1$}
				\put(-35,5){$v_2$}
				\put(10,5){$v_{n-1}$}
				\put(65,5){$v_{n}$}		
			
				\put(-70,25){\textcolor{gray}{$\alpha(e_{1,2})$}}			
				\put(20,25){\textcolor{gray}{$\alpha(e_{n-1,n})$}}
			\end{picture}	
			\end{center}
The GKM matrix for this path is
\[  \left[\begin{array}{rrrrrrr|l}
  1& -1 & 0 & 0 & \cdots & 0 & 0 &  \alpha_{1,2}\\
  0 & 1 & -1 & 0 & \cdots & 0 & 0 & \alpha_{2,3} \\
  0 & 0& 1 & -1 & \cdots & 0 & 0 & \alpha_{3,4}\\
 &&&&&&&\\
  \vdots  & \vdots  & \vdots & \vdots & \ddots  & \vdots & \vdots & \vdots\\
  &&&&&&&\\
0 & 0 & 0 & 0 & \cdots & 1 & -1 & \alpha_{n-1,n}\end{array}\right]
\]
where $\alpha_{i,i+1} \in \alpha(e_{i,i+1})$ are arbitrarily chosen. As before,  we can row-reduce the GKM matrix  by setting row $i$ to be the sum $\sum_{k=i}^{n} (\text{row }k)$ for each $1 \leq i \leq n$.  We obtain an equivalent system of rank $n-1$ in which $\spl{p}{v_n}$ is the only free variable in the set $\{\spl{p}{v_i} : i=1,\ldots,n\}$.	(This system is of maximal rank since an $(n-1) \times (n+1)$ system of equations can have at most one free variable among the $\spl{p}{v_i}$.) Figure \ref{fig:Path GKM} shows this equivalent system:
\begin{figure}[ht]
$\left[\begin{array}{rrrrrrr|l}
  1& 0 & 0 & 0 & \cdots & 0 & -1 &  \alpha_{n-1,n} + \ldots + \alpha_{3,4}+\alpha_{2,3} +\alpha_{1,2} \\
  0 & 1 & 0 & 0 & \cdots & 0 & -1 & \alpha_{n-1,n} + \ldots + \alpha_{3,4}+\alpha_{2,3} \\
  0 & 0& 1 & 0 & \cdots & 0 & -1 & \alpha_{n-1,n}+ \ldots + \alpha_{3,4}\\
 &&&&&&&\\
  \vdots  & \vdots  & \vdots & \vdots & \ddots  & \vdots & \vdots & \vdots\\
  &&&&&&&\\
0 & 0 & 0 & 0 & \cdots & 1 & -1 & \alpha_{n-1,n}\end{array}\right]$
 \caption{A system equivalent to the GKM matrix for $P_{n}$}
   \label{fig:Path GKM}
\end{figure}
\end{example}

The linear combinations that occur in the last column of the matrix in Figure \ref{fig:Path GKM} can be used to construct generalized splines for more complicated graphs as well.  For instance, the next result builds on this description of paths to describe a collection of (usually) nontrivial generalized splines for the cycle $C_n$.  

\begin{theorem}
\label{thm: Cycle1}
Let $C_n$ be a finite edge--labeled cycle given by vertices $v_1, v_2, \ldots, v_n$ in order. Define the vector $\spl{p}{} \in R^{|V|}$ with  
\begin{equation}
\label{cycleEq}
\left[\begin{array}{c}\spl{p}{v_1} \\\spl{p}{v_2} \\\spl{p}{v_3} \\ \vdots  \\\spl{p}{v_{n-1}} \\\spl{p}{v_n}\end{array}\right] = \spl{p}{v_1}\left[\begin{array}{c}1 \\1\\1 \\\vdots \\1 \\1\end{array}\right]+\alpha_{1,n}\begin{bmatrix}
  0 & 0  & \cdots & 0 & 0\\
  1& 0 & \cdots &0 &0 \\
  1 & 1 & \cdots & 0 & 0  \\
  \vdots  & \vdots  & \ddots & \vdots & \vdots  \\ 
  1 & 1 & \cdots  & 1 & 0\\
  1 & 1 & \cdots & 1 & 1
 \end{bmatrix}\left[\begin{array}{c}\alpha_{1,2}\\\alpha_{2,3} \\\alpha_{3,4} \\ \vdots \\ \alpha_{n-2,n-1}\\ \alpha_{n-1,n}\end{array}\right]
\end{equation}
with arbitrary choices of $\spl{p}{v_1} \in R$, $\spl{\alpha}{i,i+1} \in \alpha(e_{i,i+1})$, and $\spl{\alpha}{1, n} \in \alpha(e_{1,n})$.  Then $\spl{p}{}$ is a generalized spline for $C_n$.  The spline $\spl{p}{}$ is nontrivial exactly when $\alpha_{1,n}$ and at least one of the $\alpha_{i,i+1}$ are nonzero.
\end{theorem}

\begin{proof}
We check that $\spl{p}{} \in R^{n}$ satisfies the GKM condition at every edge of $C_n$. For all $i$ with $2 \leq i \leq n-1$ we have
\begin{align*}
\spl{p}{v_{i+1}} - \spl{p}{v_i} & = \left[\spl{p}{v_1} + \alpha_{1,n}(\alpha_{1,2} + \ldots +  \alpha_{i-1,i} + \alpha_{i,i+1})\right]  - \left[ \spl{p}{v_1} + \alpha_{1,n}(\alpha_{1,2} + \ldots + \alpha_{i-1,i})\right]\\
& = \alpha_{1,n}\alpha_{i,i+1}
\end{align*}
which is in $\alpha(e_{i,i+1})$ by assumption on $\alpha_{i,i+1}$. 
It remains to check that the GKM condition is satisfied at edges $e_{1,2}$ and $e_{1,n}$.
At edge $e_{1,2}$ we have 
\[\spl{p}{v_2} - \spl{p}{v_1} = [\spl{p}{v_1} + \alpha_{1,n}\alpha_{1,2}] - \spl{p}{v_1}  = \alpha_{1,n}\alpha_{1,2} \]
which is in the ideal $\alpha(e_{1,2})$.
At edge $e_{1,n}$ we have
\[\spl{p}{v_n} - \spl{p}{v_1}  = [\spl{p}{v_1} + \alpha_{1,n}(\alpha_{1,2} + \ldots + \alpha_{n-1,n})] -  \spl{p}{v_1}
 = \alpha_{1,n}(\alpha_{1,2} + \ldots + \alpha_{n-1,n})\]
which is in the ideal $\alpha(e_{1,n})$.  Hence $\spl{p}{}$ is a generalized spline for $C_n$.  The spline $\spl{p}{}$ is nontrivial if and only if the second term is nonzero, namely when $\alpha_{1,n}$ and at least one of the $\alpha_{i,i+1}$ are nonzero.
\end{proof}

Theorem \ref{thm: Cycle1} actually does more: it identifies a collection of generalized splines for $C_n$ that are linearly independent for many choices of $R$.  Indeed, we can write the generalized splines from Theorem \ref{thm: Cycle1}  in parametric form:
\begin{align}
\label{LinIndepent for RCn}
\left[\begin{array}{l} \spl{p}{v_1}\\ \spl{p}{v_2} \\ \spl{p}{v_3} \\ \spl{p}{v_4} \\ \vdots \\ \spl{p}{v_n} \end{array}\right] 
= \spl{p}{v_1}\left[\begin{array}{c}1 \\1 \\1 \\1 \\\vdots \\1 \end{array}\right]+
\spl{\alpha}{1,n}\spl{\alpha}{1,2}\left[\begin{array}{c}0 \\1 \\1 \\1 \\\vdots \\1 \end{array}\right]+
\spl{\alpha}{1,n}\spl{\alpha}{2,3}\left[\begin{array}{c}0 \\0 \\1 \\1 \\\vdots \\1 \end{array}\right]+
\cdots +
\spl{\alpha}{1,n}\spl{\alpha}{n-1,n}\left[\begin{array}{c}0 \\0 \\0 \\0 \\\vdots \\1 \end{array}\right]
\end{align}
with coefficients $ \spl{p}{v_1} \in R$ and $\spl{\alpha}{i,i+1}  \in \alpha(e_{i,i+1}) = I_{i,i+1}$  for all $1 \leq i \leq n-1$.    The vectors $[1,1,1,\ldots,1]^T, [0,1,1,\ldots,1]^T, \ldots, [0,0,0,\ldots,1]^T$ are linearly independent in $R^n$ but are not necessarily elements of $R_{C_n}$.  If $R$ is an integral domain then for any fixed choices of $\alpha_{i,j} \in \alpha(e_{i,j})=I_{i,j}$ the vectors $[1,1,1,\ldots,1]^T, \spl{\alpha}{1,n}\spl{\alpha}{1,2}[0,1,1,\ldots,1]^T, \ldots, \spl{\alpha}{1,n}\spl{\alpha}{n-1,n}[0,0,0,\ldots,1]^T$ are both linearly independent and in $R_{C_n}$.  

We will use these kinds of splines---which arise naturally when considering the GKM matrix---repeatedly in subsequent sections of the paper.


\begin{example} \label{example: arbitrary K_4}
We return to the case of the complete graph $K_4$ whose ideals $\alpha(e)$ are all principal.  By Definition \ref{Gen Spline Defn}, the tuple
$\spl{p}{} = (\spl{p}{v_1}, \spl{p}{v_2}, \spl{p}{v_3}, \spl{p}{v_4})$ is a generalized spline for $K_4$ if and only if it satisfies the GKM conditions in Figure \ref{fig:Arbitrary K_4}.
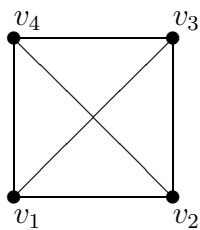
\begin{figure}[ht]
\begin{minipage}[b]{0.45\linewidth}
\flushright
\begin{picture}(100,60)(-40,0)
					\put(-30,0){\circle*{5}}
					\put(-30,60){\circle*{5}}
					\put(30,0){\circle*{5}}
					\put(30,60){\circle*{5}}
					\put(-30,0){\line(0,1){60}}
					\put(-30,0){\line(1,0){60}}
					\put(-30,60){\line(1,0){60}}
					\put(30, 0){\line(0,1){60}}
					\put(-30,0){\line(1,1){60}}
					\put(-30,60){\line(1,-1){60}}
					\put(-30, -10){$v_1$}
					\put(30, -10){$v_2$}
					\put(-30, 65){$v_4$}
					\put(30, 65){$v_3$}
			\end{picture} 
\end{minipage}
\hspace{0.5cm}
\begin{minipage}[b]{0.45\linewidth}
\flushleft
$\spl{p}{v_1}-\spl{p}{v_2}  \in \alpha(e_{1,2}) = \langle \alpha_{1,2} \rangle$\\ 
$\spl{p}{v_1}-\spl{p}{v_3} \in \alpha(e_{1,3})= \langle \alpha_{1,3} \rangle$ \\
$\spl{p}{v_1}-\spl{p}{v_4} \in \alpha(e_{1,4})= \langle \alpha_{1,4} \rangle$ \\
$\spl{p}{v_2}-\spl{p}{v_3} \in \alpha(e_{2,3})= \langle \alpha_{2,3} \rangle$ \\
$\spl{p}{v_2}-\spl{p}{v_4} \in \alpha(e_{2,4}) = \langle \alpha_{2,4} \rangle$ \\
$\spl{p}{v_3}-\spl{p}{v_4} \in \alpha(e_{3,4}) = \langle \alpha_{3,4} \rangle$ 
\end{minipage}
\caption{GKM Conditions for $K_4$ whose ideals are all principal}
\label{fig:Arbitrary K_4}
\end{figure}

The difference
$\spl{p}{v_i}-\spl{p}{v_j}$ is in the ideal $\alpha(e_{i,j}) = \langle \alpha_{i,j} \rangle$ if and only if 
$\spl{p}{v_i}-\spl{p}{v_j} = q_{i,j}\alpha_{i,j}$ for some $q_{i,j} \in R$, so we represent these GKM conditions by the following matrix equation.  (The coefficient matrix is the GKM matrix.)

\begin{equation*}
\left[ \begin{array}{rrrr} 1 & -1 & 0 & 0 \\ 1 & 0& -1 & 0 \\ 1 & 0& 0 & -1 \\ 0 & 1 & -1 & 0 \\ 0 & 1 & 0 & -1 \\ 0 & 0 & 1 & -1\end{array} \right] 
\left[ \begin{array}{r} \spl{p}{v_1} \\ \spl{p}{v_2} \\ \spl{p}{v_3} \\ \spl{p}{v_4} \end{array} \right] = 
\left[ \begin{array}{ccccc} q_{1,2}, q_{1,3}, q_{1,4}, q_{2,3}, q_{2,4}, q_{3,4} \end{array} \right ]
\left[ \begin{array}{c} \alpha_{1,2} \\ \alpha_{1,3} \\ \alpha_{1,4} \\ \alpha_{2,3} \\ \alpha_{2,4} \\ \alpha_{3,4} \end{array} \right] 
\end{equation*}

After several invertible row operations in which we add various rows to other rows, we obtain an  equivalent system of equations such as that given in Figure \ref{fig 3.2}.    


\begin{figure}[!htbp]
$M_{K_4} = \left[ \begin{array}{rrrr|l} 1 & 0 & 0 &-1& q_{1,4}\alpha_{1,4} \\ 0 & 1 & 0 & -1 & q_{2,3}\alpha_{2,3} \\ 0 & 0& 1 & -1 & q_{3,4}\alpha_{3,4}\\ 0 & 0 & 0 & 0 & q_{1,2}\alpha_{1,2} - q_{1,4}\alpha_{1,4} + q_{2,4}\alpha_{2,4}\\ 0 & 0 & 0 & 0 & q_{1,3}\alpha_{1,3} - q_{1,4}\alpha_{1,4} + q_{3,4}\alpha_{3,4} \\ 0 & 0 & 0 & 0 & q_{2,3}\alpha_{2,3} - q_{2,4}\alpha_{2,4} + q_{3,4}\alpha_{3,4} \end{array} \right]$
\caption{A system equivalent to the extended GKM matrix for $K_4$ when all ideals are principal}
\label{fig 3.2}
\end{figure} 
\end{example}


\section{Generalized Splines for Trees}
\label{sec: Trees}

We will now use the GKM matrix to describe all generalized splines for trees.  We start by describing the generalized splines for paths, which uses the same argument as trees without the notational technicalities.  

Figure  \ref{fig:Path GKM} shows a matrix that is row-equivalent to the GKM matrix for the path $(P_n, \alpha)$.  The solutions can be written in parametric form as:
$$\left[\begin{array}{l} \spl{p}{v_1}\\ \spl{p}{v_2} \\ \spl{p}{v_3} \\ \spl{p}{v_4} \\ \vdots \\ \spl{p}{v_{n-1}} \\ \spl{p}{v_n} \end{array}\right] 
= \spl{p}{v_n}\left[\begin{array}{c}1 \\1 \\1 \\1 \\\vdots \\1 \\ 1\end{array}\right]+
\spl{\alpha}{n-1,n}\left[\begin{array}{c}1 \\1 \\1 \\1 \\\vdots \\1 \\ 0 \end{array}\right]+
\cdots +
\spl{\alpha}{3,4} \left[\begin{array}{c}1 \\1 \\1 \\0 \\\vdots \\0 \\0 \end{array}\right]+
\spl{\alpha}{2,3} \left[\begin{array}{c}1 \\1 \\0 \\0 \\\vdots \\0 \\0 \end{array}\right]+
\spl{\alpha}{1,2} \left[\begin{array}{c}1 \\0 \\0 \\0 \\\vdots \\0 \\0 \end{array}\right]$$
where the coefficients $ \spl{p}{v_n}$ and $\alpha_{i,i+1}$  for all $1 \leq i \leq n-1$ are chosen arbitrarily from the sets $R$ and $\alpha(e_{i,i+1}) = I_{i,i+1}$ respectively.  By Corollary \ref{corollary: solution set is gen splines}, this gives precisely the collection of generalized splines for the path $P_n$.

When $R$ is an integral domain, this also gives linearly independent vectors in $R_{P_n}$ (for any choices of $\alpha_{i,i+1} \in I_{i,i+1}$):
\begin{align}
\label{Basis for RPn}
\mathcal{B}_{R_{P_n}} = 
\left\{ 
\left[\begin{array}{c}1 \\1 \\1 \\1 \\\vdots \\1 \\ 1\end{array}\right] , 
\left[\begin{array}{c}\alpha_{n-1,n} \\\alpha_{n-1,n} \\\alpha_{n-1,n} \\\alpha_{n-1,n} \\\vdots \\\alpha_{n-1,n} \\ 0 \end{array}\right], 
\cdots,
\left[\begin{array}{c}\alpha_{3,4} \\\alpha_{3,4} \\\alpha_{3,4} \\ 0 \\\vdots \\ 0 \\ 0 \end{array}\right], 
\left[\begin{array}{c}\alpha_{2,3} \\\alpha_{2,3} \\ 0 \\ 0 \\\vdots \\ 0 \\ 0 \end{array}\right], 
\left[\begin{array}{c}\alpha_{1,2} \\ 0\\ 0 \\ 0 \\\vdots \\ 0 \\ 0 \end{array}\right] 
\right \}
\end{align}

Morally speaking, this decomposition describes something very close to a basis for the generalized splines---as long as we can write a basis for the ideals $I_{i,i+1}$.  For instance, when each ideal $I_{i,i+1}$ is principal and $\alpha_{i,i+1}$ denotes the generator of  $I_{i,i+1}$ for each $1 \leq i \leq n-1$, then these vectors form a basis for $R_{P_n}$.  In general, we won't be able to find a basis for $R_G$ because we can't even necessarily find bases for the ideals $I_{i,i+1}$.  Even when $R$ is a polynomial ring, we need all of the technical tools developed in the theory of Gr\"{o}bner bases to compute bases of ideals in $R$.    

However, we can find generators for the splines on trees.  We  reformulate the essential property of this basis from the point of view of trees.  Observe that  $\spl{p}{} \in R_{P_n}$ must satisfy the following property for any $v_i, v_j \in V(P_n)$ with  $i < j$:
\begin{equation}
\label{GenSplinePath}
 \spl{p}{v_j} =  \spl{p}{v_i} + \displaystyle \sum_{k = i}^{j-1} \spl{\alpha}{k,k+1} \text{ for some }  \spl{\alpha}{k,k+1} \in I_{k,k+1}.
 \end{equation}

Trees are more complicated than paths, so describing the general result precisely is more complicated.  The main idea is similar to above, though.  It relies on the fact that there is exactly one path between any two vertices in a tree, as well as on Equation \eqref{GenSplinePath}.  

\begin{theorem}
\label{Tree}
Let $T = (V, E, \alpha)$ be a finite edge--labeled tree.  The tuple $\spl{p}{} \in R^{|T|}$ is a generalized spline $\spl{p}{} \in R_T$  if and only if given any two vertices $v_i, v_j \in V$ we may write
\begin{equation}
\label{treeCondition}
\spl{p}{v_j} = \spl{p}{v_i} + \spl{\alpha}{i,i_1} + \ldots + \spl{\alpha}{i_{m-1},i_m} +\spl{\alpha}{i_{m},j} \text{ for some }\spl{\alpha}{l,k} \in \alpha(e_{l,k})=I_{l,k}
\end{equation}
where $v_i, v_{i_1}, \ldots, v_{i_m}, v_j$ are the vertices in the unique path connecting $v_i$ and $v_j$ in the tree $T$.
Furthermore $\spl{p}{}$ is  non-trivial if and only if at least one of the $\spl{\alpha}{l,k}$ is nonzero.
\end{theorem}

\begin{proof}
We proceed via induction on $|V|$. The base case $|V|=1$ is trivial since $E = \emptyset$. We also prove the case $|V| = 2$, namely when $T$ is a path on two vertices. Denote the vertices of $T$ by $v_1$ and $v_2$ the edge set by $E=\{e_{1,2}\}$. Now let $\spl{p}{} = (p_{v_1}, p_{v_2}) \in R^2$.  
By Definition \ref{Gen Spline Defn} we know $\spl{p}{} \in R_T$ if and only if $\spl{p}{v_1} - \spl{p}{v_2} \in I_{1,2}$.  We rewrite this as $\spl{p}{v_1} = \spl{p}{v_2} + \spl{\alpha}{1,2}$ for some choice of $\spl{\alpha}{1,2} \in I_{1,2}$.   In other words $\spl{p}{}$ is a generalized spline for $T$ if and only if $\spl{p}{}$ satisfies Equation \eqref{treeCondition} for all pairs of vertices in $V = \{v_1, v_2\}$. Furthermore $\spl{p}{}$ is non-trivial if and only if $\spl{p}{v_1} \neq \spl{p}{v_2}$ or equivalently $\spl{\alpha}{1,2} \neq 0$.

Assume  the theorem holds for every tree with at most $n$ vertices and let $T^{\prime} = (V^\prime, E^\prime, \alpha)$  with $|V^\prime| = n+1$. 
Suppose $\spl{p}{} \in R^{|V'|}$ satisfies Equation \eqref{treeCondition} for all pairs of vertices in $V^\prime$ and let $e_{h,k} \in E^\prime$ be an arbitrary edge.
Since $v_h$ and $v_k$ are adjacent in $T^\prime$ we know $\spl{p}{k} = \spl{p}{h} + \spl{\alpha}{h,k}$ for some $ \spl{\alpha}{h,k} \in I_{h,k}$ by Equation \eqref{treeCondition}.  Rewriting this condition, we obtain  $\spl{p}{k} - \spl{p}{h} \in I_{h,k}$. Since $e_{h,k}$ was arbitrary we conclude $\spl{p}{} \in R_{T^\prime}$. 

Conversely, suppose that $\spl{p}{} \in R_{T^\prime}$.  We show that $\spl{p}{}$ satisfies Equation \eqref{treeCondition} for all vertices in $V^\prime$. Without loss of generality, label the vertices of $T^\prime$ so that $v_{n+1}$ is a leaf adjacent to $v_n$.  Choose arbitrary $v_i, v_j \in V^\prime$ and let $v_i, v_{i_1}, \ldots, v_{i_m}, v_j$ denote the vertices in the unique path connecting $v_i$ and $v_j$ in $T^\prime$.  Let $T$ denote the subgraph $T \subseteq T^\prime$ induced by $v_i, v_{i_1}, \ldots, v_{i_m}, v_j$.  The graph $T$ is  a tree itself, since it is a connected subgraph of a tree.  The restriction of $\spl{p}{}$ to the vertices in $T$ is a generalized spline for $T$ by Proposition \ref{proposition: subgraphs}.  If $T$ has at most $n$ vertices then the inductive hypothesis implies that $\spl{p}{}$ satisfies Equation \eqref{treeCondition} for the pair $v_i, v_j$.   If $T$ has $n+1$ vertices then $T$ is a path of length $n+1$.  Figure \ref{fig:Path GKM} shows a system equivalent to the GKM matrix in this case.  The first row of this matrix describes the equation
\[\spl{p}{v_j} = \spl{p}{v_i} + \spl{\alpha}{i,i_1} + \ldots + \spl{\alpha}{i_{m-1},i_m} +\spl{\alpha}{i_{m},j}\]for some set $\spl{\alpha}{l,k} \in \alpha(e_{l,k})=I_{l,k}$.  In other words, this graph also satisfies Equation \eqref{treeCondition}, proving our claim.

Finally, the spline $\spl{p}{}$ is nontrivial if and only if there exist some pair $v_i, v_j \in V^\prime$ such that $\spl{p}{v_i} \neq \spl{p}{v_j}$.  This is equivalent to saying that  the coefficients $\spl{\alpha}{i, i_1}, \spl{\alpha}{i_1,i_2}, \ldots,  \spl{\alpha}{i_{m-1},i_m},\spl{\alpha}{i_{m},j}$ associated to the path $v_i, v_{i_1}, \ldots, v_{i_m}, v_j$ are not all equal to $0$, by Equation \eqref{treeCondition}. Equivalently there exists a pair $l,k$ with $\spl{\alpha}{l,k} \neq 0$ as desired. 
\end{proof}


\section{Existence of generalized splines and lower bounds on the rank of $R_G$}\label{section: submodules}
 
We now address a fundamental question: do nontrivial generalized splines exist for an arbitrary edge--labeled graph $(G,\alpha)$?  We solved this question in the case of edge--labeled cycles $(C_n, \alpha)$ in Theorem \ref{thm: Cycle1}.  The answer in that case (yes) leads naturally to a stronger result: Equation \eqref{LinIndepent for RCn} actually identifies a collection of generalized splines that are linearly independent when $R$ is an integral domain. The condition that $R$ be an integral domain is crucial, as Bowden and the third author show in forthcoming work \cite{BT}.

Similarly, we will answer the existence question for generalized splines on arbitrary $(G,\alpha)$ (yes, unless $G$ consists of a single vertex) by constructing a collection of generalized splines that are linearly independent when $R$ is an integral domain.  This provides a lower bound on the rank of $R_G$ as an $R$-module when $R_G$ is a free $R$-module, and constructs a collection of generators associated to vertices when the ideal $\alpha(e)$ is principal for each edge $e$.  All of these hypotheses are satisfied for the generalized splines used to construct equivariant cohomology and equivariant $K$-theory, where constructing bases is an important and well-studied question \cite{GuZ1}, \cite{GT}.  Geometrically, Theorem \ref{thm: Complete2} and Corollary \ref{cor: free submodule} partially extend existing results on flow-up classes in equivariant cohomology, since we broaden the class of varieties for which we can construct linearly-independent rank $n$ collections
of flow-up classes.  The result is new for equivariant $K$-theory.  We note, however, that our flow-up classes are generally not a basis for $R_G$.  

Corollary \ref{cor: free submodule} proves that each $R_G$ contains a free submodule of rank $n$ as a special (and simpler) case of Theorem \ref{thm: Complete2}.

\begin{theorem}
\label{thm: Complete2}
Let $(G, \alpha)$ be a finite edge--labeled graph.  Fix any subgraph $G'$ of $G$ and let $\spl{p}{}$ be a generalized spline for $(G', \alpha|_{G'})$.  Let $N_{G'} =  \prod_S \alpha_{i,j}$ where each $\alpha_{i,j}$ is a nonzero element of the ideal $\alpha(v_iv_j)$ and the product is taken over the set $S$ of edges incident to a vertex in $G'$ but not in $G'$. Namely
\[S = \{\alpha_{i,j}: v_iv_j \in E(G-G') \textup{ and } v_i \in V(G') \textup{ or } v_j \in V(G')\}.\] 
Then the vector $\spl{q}{}$ defined by 
$$\spl{q}{v_i} = \left\{
     \begin{array}{lr}
       N_{G'} \spl{p}{v_i} & \text{if }  v_i \in V(G')\\
       0 &  \text{if } v_i \notin V(G')
     \end{array}
   \right \}$$ 
is a generalized spline for $G$.
 \end{theorem}
 
 \begin{proof}
For each edge $v_iv_j \in E(G)$, there are three possibilities:
\begin{enumerate}
\item  Both $v_i, v_j \in V(G')$.  Then $\spl{p}{v_i}-\spl{p}{v_j}$ satisfies the GKM condition in $G'$.  Thus $\spl{q}{v_i}-\spl{q}{v_j} = N_{G'}(\spl{p}{v_i}-\spl{p}{v_j})$ satisfies the GKM condition for $v_i, v_j$ in $G$ since $\alpha(v_iv_j)$ is an ideal and $N_{G'} \in R$.
\item Neither $v_i$ nor $v_j$ is in $V(G')$.  Then the difference $\spl{q}{v_i}-\spl{q}{v_j}=0-0$ vacuously satisfies the GKM condition for $v_i, v_j$ in $G$.
 \item  Exactly one of $v_i,v_j$ is in $V(G')$.  Suppose that $v_i \in V(G')$ and $v_j \notin V(G')$.  Consider the difference $\spl{q}{v_i}-\spl{q}{v_j} = N_{G'}(\spl{p}{v_i}-\spl{p}{v_j})$.  The factor $N_{G'}$ is in the ideal $\alpha(v_iv_j)$ by definition of $N_{G'}$ and by definition of ideals.  Hence the product  $N_{G'}(\spl{p}{v_i}-\spl{p}{v_j})$ satisfies the GKM condition for $v_i, v_j$ in $G$.
 \end{enumerate}
We confirmed that the GKM condition is satisfied by $\spl{q}{}$ in all three cases and for every edge $v_iv_j \in E(G)$, as desired.  
 \end{proof}

The next corollary constructs classes that look like what are called ``flow-up" classes in geometric applications.  Given a partial order on the vertices of $G$, a {\em flow-up class} associated to the vertex $v$ is a generalized spline $\spl{p^v}{}$ so that for each vertex $u$ with $u \not > v$ the spline satisfies $\spl{p^v}{u} = 0$.  (In geometric applications, flow-up classes satisfy additional conditions as well.)  These classes occur naturally in geometric applications: the partial order comes from a suitably-generic one-dimensional torus action on the variety (and hence on the graph), and the spline is the cohomology class associated to the subvariety that flows into the vertex $v$.  The most famous examples of flow-up classes occur in flag varieties and Grassmannians, where they are known as Schubert classes and where they in fact form a basis for the ring of generalized splines (equivariant cohomology rings, in the geometric context).  

Our motivation for the next sequence of corollaries comes from these geometric applications.  In those cases, the ideals $\alpha(e)$ for each edge $e$ are principal.  If some ideals were not principal, the results that follow could be refined to construct a larger free submodule of $R_G$.

We now construct a rank-$n$ free submodule of the generalized splines for an arbitrary edge--labeled graph $(G,\alpha)$ using a collection of linearly-independent flow-up classes.  The reader interested only in the special case of this corollary could prove it directly by taking the special case when $G'$ is a single vertex.

\begin{corollary}
\label{cor: free submodule}
Let $R$ be an integral domain and $(G,\alpha)$ a connected edge--labeled graph on $n$ vertices.  Then $R_G$ contains a free $R$-submodule of rank $n$.
\end{corollary}

\begin{proof}
Enumerate the vertices in $V(G)$ as $v_1, v_2, \ldots, v_n$. For each $v_i$ define $G'_{i}$ to be the subgraph consisting of exactly vertex $v_i$. Clearly $\spl{p}{}=\spl{1}{}$ is a generalized spline for $(G'_i, \alpha|_{G'_i})$ for all $1 \leq i \leq n$. Then Theorem \ref{thm: Complete2} yields generalized splines $\{\spl{q_i}{}: i=1,\ldots,n\}$ for $G$, where $\spl{q_i}{v_j} = \delta_{ij}N_{G'_{i}}$ and $N_{G'_{i}}=\prod_{j \neq i} \alpha_{i,j}$ for arbitrarily chosen $0 \neq \alpha_{i,j} \in \alpha(v_iv_j)$. We show that this set is linearly independent in the $R$-module $R_G$. Suppose $\sum_{i=1}^n c_i \spl{q_i}{}=\spl{0}{}$ for coefficients $c_i \in R$.  For each $1 \leq j \leq n$, evaluation at $v_j$ yields 
\begin{equation}
\label{eq: linear independence}
\sum_{i=1}^n c_i \spl{q_i}{v_j}= \sum_{i=1}^n c_i \delta_{ij}N_{G'_{i}} = c_jN_{G'_{j}} = 0
\end{equation}
Since $R$ is an integral domain and each $\alpha_{i,j} \neq 0$ it follows that $N_{G'_{j}} \neq 0$ for all $j$.  Hence Equation \eqref{eq: linear independence} implies $c_j = 0$ for all $1 \leq j \leq n$ so that $\{\spl{q_i}{}: i=1,\ldots,n\}$ is linearly independent in $R_G$ and therefore spans a free $R$-submodule of rank $n$.
\end{proof}
%
%

The next corollary notes a particular choice for the scaling factor $N_{G'}$ in  Theorem \ref{thm: Complete2}  that can be useful in the kinds of examples that arise in geometric applications.  All of the hypotheses hold in typical geometric applications (equivariant cohomology with field coefficients, equivariant $K$-theory with field coefficients, and classical algebraic splines).

\begin{corollary}
Fix an edge--labeled graph $(G,\alpha)$ and a UFD $R$.  Suppose that for each edge $e$ the ideal $\alpha(e)$ is principal and choose a generator $\alpha_{i,j}$ for each edge $e=v_iv_j$.  Then for any subgraph $G'$ of $G$ we may apply Theorem \ref{thm: Complete2} by choosing 
$$N_{G'} = \textup{lcm} \{\alpha_{i,j} : v_iv_j \in E(G-G') \textup{ and } v_i \in V(G') \textup{ or } v_j \in V(G')\}$$
\end{corollary}

The next two corollaries of Theorem \ref{thm: Complete2} address particular ways to construct (nontrivial) generalized splines for $G$ from subgraphs of $G$.

\begin{corollary}
If $G$ contains any subgraph $G'$ for which $R_{G'}$ contains a nontrivial generalized spline then $R_{G}$ also contains a nontrivial generalized spline.
\end{corollary}

 \begin{example}
 \label{K4 example}
We can construct generalized splines for the edge--labeled graph $(K_4, \alpha)$ given in Figure \ref{fig:K4 spline} using these corollaries.  The vertex in the upper-left corner is $v_1$ and the rest of the vertices in clockwise order around the square are $v_2, v_3, v_4$.  Let $C_4$ denote the Hamiltonian cycle determined by ordering the vertices $v_1v_2v_3v_4$.  Let 
\[N_{C_4} = \textup{lcm}\{\alpha(v_1v_3), \alpha(v_2v_4)\} = \textup{lcm} \{\alpha_5, \alpha_6\}\]
with the labeling in Figure \ref{fig:K4 spline}.
Theorem \ref{thm: Cycle1} constructed many nontrivial generalized splines for $C_4$ including 
\[\spl{p}{} = \left[\begin{array}{c} 0 \\ \alpha(v_1v_4)\alpha(v_1v_2) \\ \alpha(v_1v_4)(\alpha(v_1v_2) + \alpha(v_2v_3)) \\ \alpha(v_1v_4)(\alpha(v_1v_2) + \alpha(v_2v_3)+ \alpha(v_3v_4))  \end{array} \right] =  \left[\begin{array}{c} 0 \\ \alpha_4\alpha_1 \\ \alpha_4(\alpha_1+\alpha_2) \\ \alpha_4(\alpha_1+\alpha_2+\alpha_3) \end{array} \right].\]
The corollaries show that the multiple $N_{C_4}\spl{p}{}$ is a generalized spline for $K_4$.
\end{example}

\begin{corollary}
\label{cor: Proper Subgraph}
Let $R$ be an integral domain.  If $G$ contains at least two vertices then $R_G$ contains a nontrivial generalized spline.
\end{corollary}

\begin{proof}
The vertex set $V$ has at least two vertices, so $V$ has a proper subset.  Let $G'$ denote a subgraph of $G$ induced by any proper subset of $V$.  Choose the unit $\spl{1}{} \in R_{G'}$ for the spline $\spl{p}{}$ in  Theorem \ref{thm: Complete2}.  The factor $N_{G'}$ is nonzero because $R$ is an integral domain.
\end{proof}

\section{Decomposing $R_G$ as an intersection}\label{section: intersection}

This section describes two ways to express $R_G$ as an intersection of rings $R_{G_i}$ for simpler graphs $G_i$.  Both are inspired by the GKM matrix, which allows us to recognize and manipulate the GKM conditions for various subgraphs of $G$.

In the first decomposition, we essentially reorganize the GKM matrix and identify the GKM matrices associated to subgraphs of $G$ inside the GKM matrix for $G$.  When these subgraphs are the edges themselves, we recover the result that the generalized splines are the intersection of the GKM conditions on all edges independently.  We can alternatively take these subgraphs to be trees, whose generalized splines we identified completely in Section \ref{sec: Trees}; this reduces the number of intersections needed to calculate $R_G$.  

In the other decomposition, we row-reduce the GKM matrix in a natural way to demonstrate that $R_G$ is the intersection of the generalized splines for a particular collection of subcycles of $G$.  This demonstrates how the combinatorial perspective can contribute to the study of generalized splines and GKM theory: cycles are subgraphs that do not arise from geometric considerations but are natural in this more general combinatorial setting.  It also reinforces Rose's results showing the importance of cycles in studying splines \cite{R1, R2}.  Handschy, Melnick, and Reinders identify a basis for generalized splines with integer coefficients over cycles in forthcoming work \cite{HMR}.  Bowden, Cao, Hagen, King, and Reinders give a simpler basis for generalized splines over cycles whose edge--labels satisfy a coprimality condition; this allows them to identify the ring structure of the generalized splines completely \cite{BHKR}.

We begin by expressing the ring of generalized splines as an intersection of generalized splines for subgraphs.

\begin{theorem}\label{theorem: union of graphs}
Let $(G, \alpha)$ be an edge--labeled graph.  Suppose $G_1, G_2, \ldots, G_k$ are a collection of spanning subgraphs of $G$ whose union is $G$, in the sense that $V(G_i)=V(G)$ for all $i$ and $\bigcup_{i=1}^k E(G_i) = E(G)$.  Let $\alpha_i = \alpha |_{G_i}$ be the edge--labelings given by restriction for each $i$.  Then 
\[R_G = \bigcap_{i=1}^k R_{G_i}.\]
\end{theorem}

\begin{proof}
Proposition \ref{proposition: subgraphs} showed that $R_G$ is contained in $R_{G'}$ for each spanning subgraph $G'$ of $G$, and in particular is contained in $R_{G_i}$ for each subgraph $G_i$.  Conversely,  suppose $\spl{p}{}$ is contained in $\bigcap_{i=1}^k R_{G_i}$.  Every edge $v_jv_k \in E(G)$ is contained in the edge set of (at least) one of the subgraphs, say $G_i$. The spline $\spl{p}{}$ is a generalized spline for $G_i$ by hypothesis, so the GKM condition is satisfied at $v_jv_k$ in $G_i$ and hence in $G$ as well.
\end{proof}

Theorem \ref{theorem: union of graphs} generalizes the definition of $R_G$.  Indeed, for each edge $e \in E(G)$,  consider the subgraph $G_e = (V(G), \{e\})$.  The ring $R_{G_e}$ is exactly the subring of $R^{|V(G)|}$ defined by applying the GKM condition at just the edge $e$.  Theorem \ref{theorem: union of graphs} says
\[R_G = \bigcap_{e \in E(G)} R_{G_e}\]
namely that the generalized splines on $G$ are formed by imposing the GKM condition on every edge of $G$ simultaneously.

The next corollary uses another common family of subgraphs: spanning trees.  We completely identified the generalized splines for trees in Theorem \ref{Tree}.  Thus, the corollary expresses the ring of generalized splines using far fewer intersections than in the original GKM formulation.  Calculating intersections of subrings is subtle, so this corollary reduces the computational complexity of identifying the ring of generalized splines.

\begin{corollary}
\label{cor: SpanningTrees}
If $G$ can be written as a union of spanning trees $T_1, T_2, \ldots T_m$ (whose edges are not necessarily disjoint) and if $\alpha_i = \alpha |_{T_i}$ is the edge--labeling given by restriction for each $i$ then 
\[R_G = \bigcap_{i=1}^m R_{T_i}.\]
\end{corollary}

Figure \ref{figure: triangle spanning trees} shows an example using the 3-cycle and principal-ideal edge--labels.  In this case $R_G$ can be expressed as the intersection of just two rings of generalized splines, each of which is completely known.  In fact,  Theorem \ref{Tree} says that the generalized splines for the two marked paths have the form $(p_1, p_1 + \alpha_{1,4}p_4 + \alpha_{2,4}p_2, p_1+\alpha_{1,4}p_4)$ and $(q_1, q_1 + \alpha_{1,2}q_2, q_1 + \alpha_{1,2}q_2+\alpha_{2,4}q_4)$ for free choices of elements $p_1, p_2, p_4, q_1, q_2, q_4 \in R$. The intersection of these two sets is $R_{C_3}$.


\begin{figure}
\begin{center}
\begin{picture}(63,70)(-10,0)
\put(10,0){\circle*{5}}
\put(30,30){\circle*{5}}
\put(10,60){\circle*{5}}

\put(30,30){\line(-2,3){20}}
\put(9,0){\line(2,3){20}}
\put(10,0){\line(2,3){20}}
\put(11,0){\line(2,3){20}}
\put(10,0){\line(0,1){60}}
\put(9,0){\line(0,1){60}}
\put(11,0){\line(0,1){60}}

\put(-5,-2){$v_4$}
\put(35,28){$v_2$}
\put(-5,58){$v_1$}

\put(25,10){$\alpha_{2,4}$}
\put(25,45){$\alpha_{1,2}$}
\put(-12,28){$\alpha_{1,4}$}
\end{picture}
\hspace{1.5in}
\begin{picture}(63,70)(-10,0)
\put(10,0){\circle*{5}}
\put(30,30){\circle*{5}}
\put(10,60){\circle*{5}}

\put(10,0){\line(0,1){60}}
\put(9,0){\line(2,3){20}}
\put(10,0){\line(2,3){20}}
\put(11,0){\line(2,3){20}}
\put(29,30){\line(-2,3){20}}
\put(30,30){\line(-2,3){20}}
\put(31,30){\line(-2,3){20}}

\put(-5,-2){$v_4$}
\put(35,28){$v_2$}
\put(-5,58){$v_1$}

\put(25,10){$\alpha_{2,4}$}
\put(25,45){$\alpha_{1,2}$}
\put(-12,28){$\alpha_{1,4}$}
\end{picture}
\end{center}
\caption{Two spanning trees whose generalized splines determine $R_{C_3}$} \label{figure: triangle spanning trees}
\end{figure}
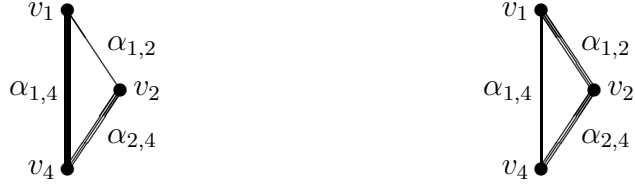

Given a connected graph $G$, we could also use Theorem \ref{theorem: union of graphs} to describe $R_G$ in terms of the generalized splines for cycles as follows.  Fix a spanning tree $T$ for $G$.  For each edge $e \in E(G) - E(T)$ let $C_e$ denote the unique cycle contained in $T \cup \{e\}$.  (This cycle exists and is unique by a classical result in graph theory  \cite[Pages 68--69]{W}.)  Let  $C_e'$ be the graph containing the cycle $C_e$ as one connected component and the rest of the vertices of $G$ as the other connected components.  Then
\begin{equation} \label{equation: intersection of cycles}
R_G = R_T \cap \bigcap_{e \in E(G) - E(T)} R_{C_e'}
\end{equation}
by Theorem \ref{theorem: union of graphs}.

However a natural row-reduction of the GKM matrix of $G$ proves this intersection directly.  To motivate our approach, we return to the complete graph on four vertices with principal-ideal edge--labels from Example \ref{example: arbitrary K_4}. The system of equations in Figure \ref{fig 3.2} is consistent precisely when $\spl{q}{}=(q_{1,2}, q_{1,3}, q_{1,4}, q_{2,3}, q_{3,4}) \in R^5$ satisfies the following homogenous system of equations:
\begin{equation}
\label{K4 Homog}
\begin{array}{l}
q_{1,2}\alpha_{1,2} - q_{1,4}\alpha_{1,4} + q_{2,4}\alpha_{2,4} = 0 \\
q_{1,3}\alpha_{1,3} - q_{1,4}\alpha_{1,4} + q_{3,4}\alpha_{3,4}  = 0 \\
q_{2,3}\alpha_{2,3} - q_{2,4}\alpha_{2,4} + q_{3,4}\alpha_{3,4} = 0 \\
\end{array}
\end{equation}
Figure \ref{figure: triangle with GKM matrices} shows the edge--labeled $3$-cycle $v_1, v_2, v_4$ of Figure \ref{figure: triangle spanning trees}, its extended GKM matrix, and a natural row-reduction of its extended GKM matrix.  
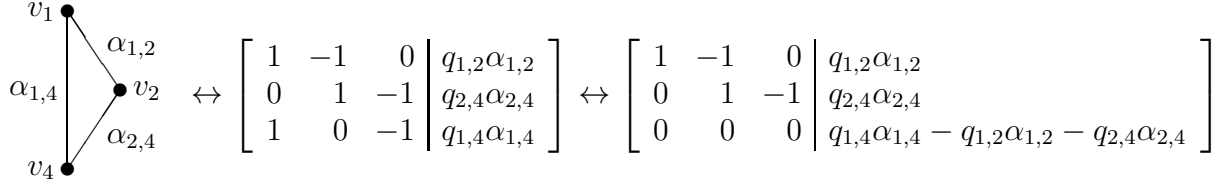
\begin{figure}
\begin{picture}(63,70)(-10,25)
\put(10,0){\circle*{5}}
\put(30,30){\circle*{5}}
\put(10,60){\circle*{5}}

\put(10,0){\line(2,3){20}}
\put(30,30){\line(-2,3){20}}
\put(10,0){\line(0,1){60}}

\put(-5,-2){$v_4$}
\put(35,28){$v_2$}
\put(-5,58){$v_1$}

\put(25,10){$\alpha_{2,4}$}
\put(25,45){$\alpha_{1,2}$}
\put(-12,28){$\alpha_{1,4}$}
\end{picture}
$\leftrightarrow
\left[ \begin{array}{rrr|l}
1 & -1 & 0 & q_{1,2}\alpha_{1,2} \\
0 & 1 & -1 & q_{2,4} \alpha_{2,4}\\
1 & 0 & -1 & q_{1,4} \alpha_{1,4}
\end{array} \right]
\leftrightarrow
\left[ \begin{array}{rrr|l}
1 & -1 & 0 & q_{1,2} \alpha_{1,2} \\
0 & 1 & -1 & q_{2,4} \alpha_{2,4} \\
0 & 0 & 0 & q_{1,4} \alpha_{1,4} - q_{1,2}\alpha_{1,2} - q_{2,4}\alpha_{2,4}
\end{array} \right]$
\caption{A triangle, its extended GKM matrix, and a row-reduction}  \label{figure: triangle with GKM matrices}
\end{figure}
The equation that remains is (up to sign) the same as that which occurs in Equation \eqref{K4 Homog}.  In fact, the entire system in Equation \eqref{K4 Homog} arises from the equations (up to sign) for the three subcycles induced by the vertices:
\begin{itemize}
\item $v_1, v_2, v_4$ 
\item $v_1, v_3, v_4$ and
\item $v_2, v_3, v_4$.
\end{itemize}


 
The next theorem generalizes this example.  Together with Remark \ref{remark: syzygies}, we also see it as a first step towards generalizing Rose's work on syzygies of edge-ideals \cite{R1, R2}.

 \begin{theorem}\label{theorem: Intersection}
Suppose that $(G,\alpha)$ is an edge--labeled graph on $n$ vertices.  Fix a spanning tree $T$ for $G$.  For each edge $e \in E(G) - E(T)$ let $C_e$ denote the unique cycle contained in $T \cup \{e\}$. 
Then the extended GKM matrix for $G$ is equivalent to an extended GKM matrix for $T$, followed  for each edge $e \in E(G) - E(T)$ by:

a row that is zero except in the last column, which is $\sum_{e' \in C_e} q_{e'}$ where $q_{e'}$ are arbitrary elements of $\alpha(e')$.
\end{theorem}


 \begin{proof}
 Choose a spanning tree $T$ for the graph $G$.  We assume without loss of generality that the first $n-1$ rows of the GKM matrix for $G$ correspond to the edges in $T$.  The first $n-1$ rows of the GKM matrix of $G$ thus consist of the GKM matrix for $T$, by construction.
 
 Consider each of the other rows in turn.  Each row corresponds to an edge $e$ in $G$ but not $T$.  
We now describe an invertible row operation to eliminate all nonzero entries from the first $n$ columns of the row corresponding to $e$ and describe $R_{G}$ more precisely.  Denote the edges of the cycle $C_e$ by $e_1=e =v_{i_1}v_{i_2}, e_2=v_{i_2}v_{i_3}, \ldots, e_k=v_{i_k}v_{i_1}$.  Let  $c_j \in \{\pm 1\}$ be the entry in the row corresponding to $e_j$ and the column corresponding to vertex $v_{i_j}$ for each $2 \leq j \neq n$.  Denote the $e_j^{th}$ row of the GKM matrix by $r_{e_j}$.  The sum of the scaled rows $\sum_{j=2}^k c_j r_{e_j}$ has $1$ in column $v_{i_2}$, $-1$ in column $v_{i_1}$, $0$ in the rest of the first $n$ columns, and $\sum_{j=2}^k c_jq_j$ in the last column, all by definition of the GKM matrix.  Finally add $\sum_{j=2}^k c_j r_{e_j}$ to the row corresponding to $e$.  This leaves $0$ in the first $n$ columns of row $e$ and $q_e + \sum_{j=2}^k c_jq_j$ in the last entry of the row.

The elements $q_e$ and $q_j$ are arbitrary elements of their respective ideals and $c_j$ is a unit in $R$ for each $j$ so the set of all possible  $q_e + \sum_{j=2}^k c_jq_j$ is the same as the set of all possible $\sum_{e' \in C_e} q_{e'}$.  The result follows.
 \end{proof}

 The last corollary uses this information to describe the generalized splines for $G$ in terms of the generalized splines for cycles, as promised.
  
  \begin{corollary}
  Suppose that $(G,\alpha)$ is an edge--labeled graph on $n$ vertices.  Fix a spanning tree $T$ for $G$.  For each edge $e \in E(G) - E(T)$ let $C_e$ denote the unique cycle contained in $T \cup \{e\}$ together with the other vertices in $G$.  Then
\[ R_G = R_T \cap \bigcap_{e \in E(G) - E(T)} R_{C_e'} \]
  \end{corollary}

\begin{proof}
Consider an edge $e$ outside of the spanning tree $T$ and its corresponding cycle $C_e$.  The previous theorem showed that the submatrix of an extended GKM matrix for $G$ given by the rows indexed by the edges $e' \in E(C_e)$ form an extended GKM matrix for the cycle $C_e$.  The vector $\spl{p}{} \in R^{|V|}$ solves an extended GKM matrix for $G$ if and only if it simultaneously solves the corresponding extended GKM matrices for $T$ and all of the $C_e$ for $e \in E(G)-E(T)$.  This proves the claim.
\end{proof}

\section{Open questions}
\label{section: open}

We end with several open questions, extending some of the major research problems for splines and GKM theory to the context of generalized splines.

Most research into what we call generalized splines focuses on particular examples, whether because of explicit hypotheses (e.g. a particular choice of the ring $R$, the graph $G$, or the edge--labeling function $\alpha$) or implicit hypotheses (e.g. that edge--labels be principal). Special cases remain very important, both for applications and for data to build the general theory.

\begin{question}
Identify $R_G$ in important special cases: for instance, when all edge--labels $\alpha(e)$ are principal ideals; or when $R$ is a particular ring (integers, polynomial rings, ring of Laurent polynomials); or when $G$ is a particular graph or family of graphs (cycles, complete graphs, bipartite graphs, hypercubes).
\end{question}

Splines on complete graphs are particularly important for approximation theory, where they appear as the Alfeld split of a simplex (for a proof see \cite[Section 3.1]{T3}).

Billera first asked the following question, seeking an interpretation of $r$-smoothness in the context of equivariant cohomology.  We extend Billera's question to ask about the analogue of $r$-smoothness for generalized splines over arbitrary rings.

\begin{question}
Let $(G,\alpha)$ be an edge--labeled graph.  Define the function $\alpha^r: E \rightarrow \mathcal{I}$ by the condition that for each edge $e$ the image $\alpha^r(e)$ is the $r^{th}$ power $\left( \alpha(e) \right)^r$.  The $r$-smooth generalized splines are the elements of the ring $R_{G, \alpha^r}$.  We ask how the $r$-smooth generalized splines compare for various $r$.  Billera asks for a geometric interpretation of $r$-smoothness in the context of equivariant cohomology rings.
\end{question}

As a module, the generalized splines $R_G$ can also be viewed as group representations: for instance, the group of automorphisms of the graph $G$ that preserve the edge--labeling naturally induces a representation on $R_G$.  Representations obtained in this and similar ways are often intrinsically interesting \cite{F, T2} and can also be a powerful tool with which to approach other questions in this section \cite{T2}.

\begin{question}
Given a specific automorphism group, what are the induced representations on $R_G$ (in terms of irreducible representations, say)? For what families of graphs are there nontrivial representations on $R_G$?
\end{question}

Propositions \ref{proposition: subgraphs} and \ref{proposition: edge deletion} and Sections \ref{section: submodules} and \ref{section: intersection} all use combinatorial aspects of graphs to analyze the ring of generalized splines.  More recently Handschy-Melnick-Reinders \cite{HMR} and Bowden-Cao-Hagen-King-Reinders \cite{BHKR} use deletion and contraction to study splines on cycles.  We believe that these are special cases of a more general relationship between the underlying combinatorics and geometry. 

\begin{question}
How do classical graph-theoretic constructions (like deletion-contraction) affect the algebraic structure of splines $R_G$?
\end{question}

Theorem \ref{thm: module direct sum},  Theorem \ref{Tree}, and Theorem \ref{thm: Complete2} are part of a larger program to identify useful bases for splines and GKM modules \cite{H, GT, GuZ2}.  The next question extends that program to generalized splines.

\begin{question}
Given a graph $G$, find a minimal generating set (or basis, if $R$ is an integral domain) for the generalized splines $R_G$.  If $G$ is a particular family of graphs (cycles, complete graphs, etc.), can we find a minimal generating set (or basis) for $R_G$? 
\end{question}

More specifically, geometers think about bases with particular ``upper-triangularity" properties that arise in many important examples, like Schubert classes, Bialynicki-Birula classes, and the canonical classes of Knutson--Tau \cite{KT} and Goldin--Tolman \cite{GT} (see also work of Harada--Tymoczko \cite{HT}).  Theorem \ref{thm: Complete2} is an initial step in constructing {\em flow-up bases} for generalized splines.

\begin{question}
What is the right definition for a flow-up class in the module of generalized splines?  Under what conditions is there a flow-up basis for the generalized splines?
\end{question}

Answering the previous question may require further extending generalized splines so that the vertices are labeled by different modules $M_v$ rather than a fixed ring $R$, as described in the Introduction to this paper.  Characterizing those splines would have immediate implications in geometric applications like computing equivariant intersection homology.

\begin{question}
Which of the results in this paper extend to generalized splines over modules?  Is there an algorithm or an explicit formula to construct flow-up basis classes for generalized spines over modules?
\end{question}


\bibliographystyle{amsalpha}

\end{document}